\documentclass{amsart}
\usepackage{amssymb}
\usepackage{bbm}
\usepackage[all]{xy}
\usepackage{amsfonts}
\usepackage{amsfonts}
\usepackage{amsfonts}
\usepackage{amscd}
\usepackage{mathrsfs}
\usepackage{amsmath}
\usepackage{graphicx}
\usepackage{subfigure}
\usepackage[english]{babel}
\usepackage{times}
\usepackage[T1]{fontenc}
\usepackage{ctex}
\usepackage{enumerate}
\usepackage{enumitem}
\usepackage{amssymb}
\usepackage{bbm}
\usepackage{amsfonts}
\usepackage{amsfonts}
\usepackage{amsfonts}
\usepackage{amscd}
\usepackage{mathrsfs}
\usepackage{amsmath}
\usepackage{bm}

\numberwithin{equation}{section}

\newtheorem{theorem}{Theorem}[section]
\newtheorem{pro}[theorem]{Proposition}
\newtheorem{exa}[theorem]{Example}
\newtheorem{cor}[theorem]{Corollary}
\newtheorem{lem}[theorem]{Lemma}
\newtheorem{dfn}[theorem]{Definition}
\theoremstyle{remark}
\newtheorem{rem}[theorem]{Remark}

\newcommand{\Z}{\mathbb{Z}}
\newcommand{\LL}{\mathcal{L}}
\newcommand{\p}{\partial}
\newcommand{\M}{\lambda}
\newcommand{\C}{\mathbb{C}}
\newcommand{\N}{\mathbb{N}}
\begin{document}
		
\title{classification of  uniformly bounded simple Lie conformal algebras with upper bound one}

\author{ Maosen Xu,  Huangjie Yu, Lipeng Luo }
\address{Department of Mathematics informations, Shaoxing University, Shaoxing, 312000, P.R.China} \email{390596169@qq.com}
\address{Department of Mathematics informations, Shaoxing University, Shaoxing, 312000, P.R.China} \email{1498190639@qq.com}
\address{Department of Mathematics, Zhejiang University of Science and Technology, Hangzhou, 310023, P.R.China.}
\email{luolipeng@zust.edu.cn}

\thanks{The third author is corresponding author. This work was supported by the National Natural Science Foundation of China
	(Nos. 12171129, 12271406, 12471027, 12401035) and the Natural Science Foundation of Shanghai (No. 24ZR1471900).}


\begin{abstract}In this paper, we prove that  uniformly bounded simple Lie conformal algebra must be finitely generated. Furthermore, we give a completely classification of uniformly bounded simple Lie conformal algebras with upper bound one. 
\end{abstract}
\maketitle

\section{introduction}
 The notion of Lie conformal algebra was initially   introduced to provide an axiomatic 
properties of the operator product expansion in \cite{K}. There are many other fields closely related to Lie conformal algebras such as vertex algebras, linearly compact Lie algebras and Gelfand-Dorfman algebras. 

 In \cite{BDK}, the notion of Lie conformal algebra is explained as the  Lie algebra over a pseudo-tensor category $\mu^*(H)$. Therefore, some scholars hope to obtain Lie conformal algebra versions of important theorems in  complex Lie algebras. At present, Lie conformal algebra version of Lie theorem and  Cartan-Jacobson's theorem have been proved in \cite{DK}. In addition, P. Kolesnikov  show that the Ado theorem of Lie conformal algebra with Levi decomposition also holds\cite{PK}. In addition, D'andrea and Kac classified  any  finite simple Lie conformal algebra must be either isomorphic to Virasoro Lie conformal algebras $Vir$ or to current Lie conformal  algebra $Cur\mathfrak{g}$, where $\mathfrak{g}$ is a complex simple Lie algebra. However, it is basically impossible  to classify all infinite simple Lie conformal algebras without any restriction.
   We say a  Lie conformal algebra is uniformly bounded if it  is a $\Z$-graded Lie conformal algebra such that the rank of each component is bounded by some positive integer. Besides the current Lie conformal algebras, there are some other types of uniformly bounded simple  Lie conformal algebras are found. The Loop Virasoro Lie conformal algebra 
\[ \mathcal{V}(s)=\bigoplus_{i \in \mathbb{Z}}\mathbb{C}[\partial]L_i, s\in \mathbb{C}:  \text{for}\ \  i,j\in \mathbb{Z},\ \ [{L_i}_\lambda L_j]=(\partial+2\lambda+s(i-j))L_{i+j}.\]
 can be found in \cite{WCY}. 
    Further, through some simple Novikov algebras and Gel'fand-Dorfman algebras, Hong and Wu obtain  the  Lie conformal algebras
   \begin{enumerate}[fullwidth,itemindent=0em,label=(\arabic*)]
   \item[$\bullet$] $CL_1(s)=\bigoplus_{i \in \mathbb{Z}_{\geq -1}}\mathbb{C}[\partial]L_i: [{L_i}_\lambda L_j]=((i+1)\partial+(i+j+2)\lambda+s(i-j))L_{i+j},$
   \item[$\bullet$] $CL_2(b,s)=\bigoplus_{i \in \mathbb{Z}}\mathbb{C}[\partial]L_i:[{L_i}_\lambda L_j]=((i+b)\partial+(i+j+2b)\lambda+s(i-j))L_{i+j}$.
   When $2b\in \Z$, $CL_2(b,s)$ is not simple. It has the unique proper graded ideal:  $SCL_2(b,s)=\bigoplus_{i\neq -2b}\mathbb{C}[\partial]L_i \oplus \mathbb{C}[\partial](\partial+2s)L_{-2b}$\end{enumerate}
   Further, the following classification results established in \cite{X}.
      \begin{pro}\label{Xu}Suppose that $\LL=\bigoplus_{i\in \mathbb{Z}}\LL_i$ be a simple graded Lie conformal algebra with upper bound one. If $\LL_{0}\cong Vir$, then $\LL$ must be isomorphic to one of follows:
   	\[SCL_2(b,s),\  \mathcal{V}(s),\ CL_2(b_1,s),\  Vir,\ CL_1(s) \]
   	where $b,b_1,s\in \C$ and $2b\in \Z$, $2b_1\notin \Z$.
   \end{pro}
  The nature question  is that can we give a complete classification with out the restriction on $\LL_0$ in Proposition \ref{Xu}.
   Using some ideals from classification of simple graded Lie conformal algebras of finite growth\cite{M}\cite{M1}\cite{M2}, the main results of this paper is  the following theorem.  
   \begin{theorem}{\label{t01}}Suppose that $\LL=\bigoplus_{i\in \mathbb{Z}}\LL_i$ be a simple graded Lie conformal algebra with upper bound  one. Then $\LL$ must be isomorphic to one of follows:
   	\[SCL_2(b,s),\  Cur\mathfrak{g},\  \mathcal{V}(s),\ CL_2(b_1,s),\  Vir,\ CL_1(s) \]
   	where $b,b_1,s\in \C$ and $2b\in \Z$, $2b_1\notin \Z$ and $\mathfrak{g}$ is a uniformly bounded simple Lie algebras with upper bound  one.
   \end{theorem}

     The paper is organized as follows.
   In Section 2,  the basic notions and propositions of ($\Z$-graded) Lie conformal algebras  are recalled. The notion of basic Lie algebra for a Lie conformal algebra is introduced. In addition, we introduce the  $\Z$-graded Lie conformal algebras involved in our classification results.  

In Section 3, we prove that any uniformly bounded simple Lie algebra must be finitely generated. As consequence, we  prove the existence of $0$-components for a  uniformly bounded simple Lie algebra.

In Section 4, we first classify the following $\Z$-graded Lie conformal algebra $\mathcal{L}=\bigoplus_{i\in \Z}\mathcal{L}_i$ provided that
\begin{enumerate}[fullwidth,itemindent=0em,label=(\arabic*)]
	\item[$({\bf C1})$]$rank \LL_i\leq 1$ for each $i \in \Z$. $rank \LL_1=rank \LL_{-1}=rank \LL_0=1$. 
	\item[$({\bf C2})$]$\LL$ is generated by $\LL_1$,$\LL_0$,$\LL_{-1}$. 
	\item[$({\bf C3})$]  $[{\LL_{-1}}_\lambda \LL_1]\neq 0,\   [{\LL_{0}}_\lambda \LL_1]\neq 0,\  [{\LL_0}_\lambda \LL_{0}]=0.$ 
\end{enumerate}
Finally, we prove that if $\LL$  is simple graded, then it must either isomorphic to current Lie conformal algebra or $SCL_2(0,s)$. 

    Through this paper, denote $\mathbb{C}$ and $\mathbb{C}^*$ the sets of all complex numbers and all nonzero complex numbers respectively. $\mathbb{R}$ is the set of real numbers. $mathbb{N}$ is the set nature number. In addition, all vector spaces and tensor products are over $\mathbb{C}$. For any vector space $V$,
    we use $V[\lambda]$ to denote the set of polynomials of $\lambda$ with coefficients in $V$.  Let $f(\partial, \lambda)\in \mathbb{C}[\lambda,\partial]$. We denote the total degree of $f(\partial, \lambda)$ by $deg f(\partial, \lambda)$, and the degree of $\lambda$ in $f(\partial, \lambda)$ by $deg_\lambda f(\partial, \lambda)$. 

 \section{Preliminary}
 In this section, we recall some basic definitions and results of Lie conformal algebras and their modules. These facts can be found in \cite{DK} and \cite{K}.
 \begin{dfn}A Lie conformal algebra is a $\C[\p]$-module $\mathcal{S}$, endowed with a family of $\C$-bilinear products $(n)$, for each $n\in \Z_+$, satisfying following axioms
\begin{enumerate}[fullwidth,itemindent=0em,label=(\arabic*)]
	\item[(C1)]$a_{(n)}b=0$ for $n\gg 0$,
	\item[(C2)]$(\partial a)_{(n)}b=-n a_{(n - 1)}b,\quad (\partial a)_{(n)}b = \partial(a_{(n)}b)+n a_{(n - 1)}b$.
	\item[(C3)]$a_{(n)}b=-\sum_{j=0}^{\infty}(-1)^{n+j}\frac{\partial^j}{j!}(b_{(n+j)}a)$.
	\item[(C4)] $a_{(m)}(b_{(n)}c)-b_{(n)}(a_{(m)}c)=\sum_{j = 0}^{\infty}\binom{m}{j}(a_{(j)}b)_{(m + n - j)}c$.
\end{enumerate}
\end{dfn}
Let $\mathcal{S}$ be a Lie conformal algebra. Regard $\mathbb{C}$ as trivial $\mathbb{C}[\partial]$-module. One can check directly that 0-th product provides a Lie algebra structure over $\mathbb{C}\otimes_{\mathbb{C}[\partial]}\mathcal{S}$.  $\mathbb{C}\otimes_{\mathbb{C}[\partial]}\mathcal{S}$ is called as the \emph{basic Lie algebra} for $\mathcal{S}$, which is denoted by $\mathcal{S}^0$. 
We can use the $\M$-brakets to describe a Lie conformal algebra in a more conciser way. 
Explicitly, let $\mathcal{S}$ be a Lie conformal algebra,	
\[[\cdot_\M \cdot]: \mathcal{S}\otimes \mathcal{S}\mapsto \mathcal{S}[\M], \] 
\[ [a_{\M}b]=\sum_{i\geq 0}\frac{\M^i}{i!}(a_{(i)}b).\]
Using $\M$-brackets, conditions (C1)-(C4) shall translate into the following axioms:
\begin{equation*}
		\begin{aligned}
			&[\partial a_\lambda b]=-\lambda[a_\lambda b], \   [a_\lambda \partial b]=(\partial+\lambda)[a_\lambda b],\ \ (conformal\   sequilinearity),\\
			&[a_\lambda b]=-[b_{-\lambda-\partial}a],\ \ (skew\ symmetry),\\
			&[a_\lambda [b_\mu c]]=[[a_\lambda b]_{\lambda+\mu}c]+[b_\mu[a_\lambda c]]\  \ (Jacobi\ identity),
		\end{aligned}
	\end{equation*}
		for all $a$, $b$, $c\in \mathcal{S}$.

	the modules, ideals, quotients, subalgebras and simplicity of a Lie conformal algebra can be defined in  canonical ways. Using conformal sequilinearity, we can define a Lie conformal algebra and its modules by giving the $\lambda$-brackets on its generators over $\mathbb{C}[\partial]$. We say a Lie conformal algebra $\mathcal{S}$ is \emph{finite} when it is finitely generated as  a $\C[\p]$-module and is \emph{finite type} when it is finitely generated as an algebra. 
	
	In addition, let $M$ be an $\mathcal{S}$ module. For any $U \subset \mathcal{S}$, $V\subset M$,  we set \[U_{(\lambda)}V:=span_{\C}\{u_{(n)}v| u\in U,v\in V,n\in \Z_+\}. \] In particular, if $M=\mathcal{S}$ is the adjoint module, we shall use the notation $[U_{(\lambda)}V]$. 
 A Lie conformal algebra $\mathcal{S}$ is called  \emph{perfect} if 
    $[\mathcal{S}_{(\lambda)} \mathcal{S}]=\mathcal{S}$.
 clearly, Simple Lie conformal algebra must be perfect.

	\begin{exa}
		The Virasoro Lie conformal algebra  $Vir=\mathbb{C}[\partial]L$ is a free  $\mathbb{C}[\partial]$-module of rank one, whose $\lambda$-brackets are determined by  $[L_\lambda L]=(\partial+2\lambda)L$. Hence $Vir^0$ is a one dimensional Lie algebra.   Furthermore, it is well known that $Vir$ is a simple Lie conformal algebra. Any free rank one $Vir$ module must be has the form $M_{a,b}=\C[\p]v$ such that
		\[L_\M v=(\p+a\M+b)v,\] for some $a,b\in \C$.
	\end{exa}
	
	\begin{exa}For a Lie algebra $\mathfrak{g}$, the current Lie conformal algebra $\text{Cur}\mathfrak{g}$ is a free $\mathbb{C}[\partial]$-module $\mathbb{C}[\partial] \otimes \mathfrak{g}$  equipped with  $\lambda$-brackets:
		\[ [x_\lambda y]=[x,y],\ \ \text{for}\ \  x,y \in \mathfrak{g}. \]
	It is easy to see that $(\text{Cur}\mathfrak{g})^0\cong \mathfrak{g}$ as Lie algebras.   
	In addition, for any $\mathfrak{g}$module $U$, we have a  $\text{Cur}\mathfrak{g}$-module $M_U:=\mathbb{C}[\partial] \otimes U$ such that
	    \[ x_\M u=x\cdot u.\]
	    for each $x\in \mathfrak{g}$ and $u\in U$.
	\end{exa}
	It is well known that any finite simple Lie conformal algebra must be isomorphic either to $Vir$ or to $Cur\mathfrak{g}$ for some finite dimensional simple Lie algebra.
	
	For a Lie conformal algebra $\mathcal{S}$, set
	\[Lie(\mathcal{S}):=\mathcal{S}[t,t^{-1}]/(\p+\p_t)\mathcal{S}[t,t^{-1}].\]
Let $a_m$ be the image of $at^n$ in  $Lie(\mathcal{S})$. We can define a Lie brackets on $Lie(\mathcal{S})$ as follows
\[[{a_m},b_n]=\sum_{j\geq 0}\binom{m}{j}(a_{(j)}b)_{m+n-j}.\]
$Lie(\mathcal{S})$ is called as the formal distribution Lie algebra of the Lie conformal
algebra $\mathcal{S}$, the following subalgebra 
\[\overline{Lie(\mathcal{S})} =\{a_n|n\in \Z_+\}\]
of $Lie(\mathcal{S})$ are called as the annihilation Lie algebra. Define the action of $\p$ on $\overline{Lie(\mathcal{S})}$ as follows as 
$\p(a_n)=na_{n-1}$, we can obtain the extended annihilation Lie algebra  $\C\p\ltimes \overline{Lie(\mathcal{S})}$.

One can check directly that $Lie(Vir)\cong \mathcal{W}$ and $Lie \mathcal(Vir)^{\_}\cong \mathcal{W}_1$, where $\mathcal{W}$ is the Witt algebra and $\mathcal{W}_1$ is the  centerless Virasoro Lie algebra.
Let $M$ be a Lie conformal algebra $\mathcal{S}$, then $M$ admits a $\C\p\ltimes \overline{Lie(\mathcal{S})}$-module structure as follows:
For any $a_n\in \overline{Lie(\mathcal{S})}$, $v\in M$. 
   \[a_n\cdot v=a_{n}v.\]
In addition, the $\partial$ of $\C\p\ltimes \overline{Lie(\mathcal{S})}$ action on $M$ is concide with the $\C[\p]$-module structure on $M$.
\begin{dfn} Let $\mathcal{S}$  be a Lie conformal algebra and $M$  be an $\mathcal{S}$ module. Then we say that $M$ is   {\emph{nilpotent}} if there exists positive integer $n$ such that for each $a_1,a_2,\cdots,a_n  \in \mathcal{S}$ and $m\in M$,\[{a_1}_{(\M_1)}{a_2}_{(\M_2)}\cdots {a_n}_{(\M_n)}m=\{0\}. \]
We say that  $M$ is {\emph{locally nilpotent}} if any  $m\in M$, there exists positive integer $n$ such that $ad^n(a_{(\M)})m=\{0\}$ for each $a\in S$.
\end{dfn}
Clealy a nilpotent module $M$ must be locally nilpotent.
   \begin{dfn}
    	We say a  Lie conformal algebra $\mathcal{G}$ is $\Z$-graded if  \begin{enumerate}[fullwidth,itemindent=0em,label=(\arabic*)]
	\item $\mathcal{G}=\bigoplus_{i\in \Z} \mathcal{G}_i$, where each $\mathcal{G}_i$ is a $\C[\partial]$-module.
	\item $[{\mathcal{G}_i}_{(\M)}{\mathcal{G}_j}]\subset \mathcal{G}_{i+j}$ for each $i,j\in \Z$.
\end{enumerate}
\end{dfn}
A graded-module of a graded-Lie conformal algebra can be defined in a similar way.
Let $\mathcal{G}=\bigoplus_{i\in \Z} \mathcal{G}_i$ be a $\Z$-graded Lie conformal algebra.
\begin{enumerate}[fullwidth,itemindent=0em,label=(\arabic*)]
\item If there exists some positive integer $N$ such that $rank(\mathcal{G}_i)<N$, then we say $\mathcal{G}$ is \emph{uniformly bounded}.
\item We say an element  $x\in \mathcal{G}$ is \emph{homogeneous} if $x\in \mathcal{G}_i$ for some $i\in \Z$. 
\item A \emph{graded ideal} of $\mathcal{G}$ is an ideal of $\mathcal{G}$,generated by homogeneous element. In addition, if $\mathcal{G}$ has no non-trivial graded ideal, then $\mathcal{G}$ is called \emph{simple graded}.
\end{enumerate}
     
In the last part of this section, Let us introduce some uniformly bounded Lie conformal algebras which will be used in our classification process. 
The following graded Lie conformal algebras are called of \emph{type I}:
\begin{enumerate}[fullwidth,itemindent=0em,label=(\arabic*)]
\item[$\bullet$]$Cur\mathfrak{g}$, where $g$ is a simple graded. 
\item[$\bullet$] $M(1)=\bigoplus_{i \geq -1}\mathbb{C}[\partial]L_i$, for each $i,j\geq -1$, 
\begin{align}\label{eq2-8}
	\begin{split} 
		[{L_i}_\lambda L_{j}]=\left\{ 
		\begin{array}{ll}
			(\partial+2\lambda)L_{2},&i=j=1,\\
			\frac{(j+1)(j-2)}{2}L_{j-1},& i=-1,j\geq{1},\\
			-jL_j,&i=0,\\
			L_{j+1},&i=1,j\geq 2,\\
			0,&i\geq2,j\geq2.
		\end{array}
		\right.
	\end{split}
\end{align}
The left unannouced $\lambda$-brakets can be obtained by skew symmetric property.

\item[$\bullet$] 
 $M(2)=\bigoplus_{i \in \mathbb{Z}}\mathbb{C}[\partial]L_i$, for each $i,j\in \Z$,\\
\begin{align}\label{eq2-9}
	\begin{split} 
		[{L_i}_\lambda L_{j}]=\left\{ 
		\begin{array}{ll}
			(\partial+2\lambda)L_{-2},& i=j=-1,\\
			L_{j-1},&i=-1,j\leq -2,\\
			\frac{(j+1)(j-2)}{2}L_{j-1},& i=-1,j\geq 1,\\
			-jL_j,&i=0,\\
			(\partial+2\lambda)L_{2},&i=j=1,\\	
			\frac{(j-1)(j+2)}{2}L_{j+1},& i=1,j\leq -2 ,\\
			L_{j+1},&i=1,j\geq 2,\\
			0,&\lvert i\rvert \geq2,\lvert j\rvert \geq2.
		\end{array}
		\right.
	\end{split}
\end{align}
The left unannouced $\lambda$-brakets of $M(1)$,$M(2)$ can be obtained by skew symmetric property.
\end{enumerate}
The following graded Lie conformal algebras are called of \emph{type II}:
\begin{enumerate}[fullwidth,itemindent=0em,label=(\arabic*)]
\item[$\bullet$]  $\mathcal{V}(s)=\bigoplus_{i \in \mathbb{Z}}\mathbb{C}[\partial]L_i, s\in \mathbb{C}:  \text{for}\ \  i,j\in \mathbb{Z},\ \ [{L_i}_\lambda L_j]=(\partial+2\lambda+s(i-j))L_{i+j}$.
\item[$\bullet$] $CL_1(s)=\bigoplus_{i \in \mathbb{Z}_{\geq -1}}\mathbb{C}[\partial]L_i: [{L_i}_\lambda L_j]=((i+1)\partial+(i+j+2)\lambda+s(i-j))L_{i+j}$.
\item[$\bullet$] $CL_2(b,s)=\bigoplus_{i \in \mathbb{Z}}\mathbb{C}[\partial]L_i:[{L_i}_\lambda L_j]=((i+b)\partial+(i+j+2b)\lambda+s(i-j))L_{i+j}$.
When $2b\in \Z$, $CL_2(b,s)$ is not simple. It has the unique proper graded ideal:  $SCL_2(b,s)=\bigoplus_{i\neq -2b}\mathbb{C}[\partial]L_i \oplus \mathbb{C}[\partial](\partial+2s)L_{-2b}$. 
\item[$\bullet$] $CL_3(s)=\bigoplus_{i \in \Z}\C[\partial]L_i:$  
 \begin{align*}
\begin{split} 
[{L_i}_\lambda L_{j}]=\left\{ 
\begin{array}{ll}
   i(\p+s)L_0,& i+j=0,\\
   -j(-\lambda+2s)L_j,&i=0,\\
   i(\p+\M+2s)L_i,&j=0,\\
   (i\p+(i+j)\M+s(i-j))L_{i+j},&else.
\end{array}
\right.
\end{split}
\end{align*}
\item[$\bullet$] $ECL(s)=\bigoplus_{i \in \mathbb{Z}}\mathbb{C}[\partial]L_i$:
	\begin{align*}
	\begin{split} 
		[{L_i}_\lambda L_{j}]=\left\{ 
		\begin{array}{ll}
			i(\partial+s)(\partial+2s)L_0,& i+j=0,\\
			-jL_j,&i=0,\\
			iL_i,&j=0,\\
			(i\partial+(i+j)\lambda+s(i-j))L_{i+j},&\text{else}.
		\end{array}
		\right.
	\end{split}
\end{align*}
\end{enumerate}
  
It is not hard to see that $ECL(s)$ has three proper ideals: \[ \bigoplus_{i\neq 0}\mathbb{C}[\partial]L_i \oplus \mathbb{C}[\partial](\partial+s) L_0, \ \ \bigoplus_{i\neq 0}\mathbb{C}[\partial]L_i \oplus \mathbb{C}[\partial](\partial+2s) L_0,\]
\[ \bigoplus_{i\neq 0}\mathbb{C}[\partial]L_i \oplus \mathbb{C}[\partial](\partial+s)(\p+2s) L_0.\]  one can check straightly that these three ideals are isomorphic to $CL_2(0,s)$,$CL_3(s)$ and $SCL_2(0,s)$, respectively.

   It is clear that the Lie conformal algebra of type I: \[Cursl(2,\C),\  M(1),\  M(2) \]
   are non-isomorphic to each other. As for the following Lie conformal algebras of type II:
   	\[ CL_2(0,s),\ CL_3(s),\ ECL(s),\ SCL_2(0,s),\]
   by	considering the action of $L_0$, the four distinct types of Lie conformal algebras are non-isomorphic to one another. Furthermore, 
   \begin{pro}\begin{enumerate}
   		\item $CL_2(0,s)\cong CL_2(0,s_1)$ if and only if $s=s_1$.
   		\item $CL_3(s)\cong CL_3(s_1)$ if and only if $s=s_1$.
   		\item $ECL(s)\cong ECL(s_1)$if and only if $s=s_1$.
   		\item $SCL_2(0,s)\cong SCL_2(0,s_1)$ if and only if $s=s_1$.
   	\end{enumerate}
   \end{pro}

\section{the existence of $\mathcal{L}_0$ for  uniformly bounded simple Lie conformal algebras}
In this section, we shall prove that a uniformly bounded simple Lie conformal algebra  $\LL=\bigoplus_{i\in \mathbb{Z}}\LL_i$ must be finitely generated with  non-trivial  $\mathcal{L}_0$.  

\begin{dfn}Suppose that $\LL=\bigoplus_{i\in \mathbb{Z}}\LL_i$ is a graded Lie conformal algebra and $M=\bigoplus_{i\in \mathbb{Z}}M_i$ is a graded $\LL$-module. Then we say $M$ is {\bf weightless} if each $M_i$ is nilpotent as a $\LL_{0}$-module.  Then we say $\LL$ is {\bf rootless} if each $\LL_i$ is nilpotent as a $\LL_{0}$-module in  terms of adjoint action.
\end{dfn}

\begin{pro}{\label{p3.2}}Let $\LL=\bigoplus_{i\in \mathbb{Z}}\LL_i$ be a perfect rootless Lie conformal algebra. Then $\LL$ is not finite type. \end{pro}
\begin{proof}Suppose  there exists  $p<0<q \in \Z$ such that $\LL$ is generated by $\LL(p,q)$, where $\LL(p,q)=\bigoplus_{p\leq i \leq q} \LL_i$.
Let  \[ \LL^+(p,q) =\bigoplus_{0<i \leq q} \LL_i\ \   \text{and}\ \  \LL^-(p,q) =\bigoplus_{p\leq i<0} \LL_i.\]
Let $\LL^+$(respectively $\LL^-$) be the subalgebra generated by $\LL^+(p,q)$(respectively $\LL^-(p,q)$).
Let $L^{+,N}$(respectively, $\LL^{-,N}$) be the $\C[\partial]$-submodule of $\LL^+$(respectively, $\LL^{-}$) with degree less than $N$(greater than $-N$). For each $N\geq 0$, one can use the Jacobi-identity and  induction on  $N$  to prove that
    \[[\LL^+(p,q)_{(\lambda)} \LL^{-,N}] \subset \LL_0+\LL^{-,N}+\LL^+(p,q),\]
        \[[\LL^-(p,q)_{(\lambda)} \LL^{+,N}] \subset \LL_0+\LL^{+,N}+\LL^-(p,q).\]
     It implies that $\LL^++\LL_0+\LL^-$ is a subalgebra of $\LL$ and generated by $\LL(p,q)$. Hence $\LL=\LL^++\LL_0+\LL^-$. It implies that $\LL^+=\bigoplus_{i>0} \LL_i$. Since $\LL^+$ is generated by  $\LL^+(p,q)$, $V=\LL^+/[\LL^+_{(\lambda)} \LL^+]$ is finite and has a nature $\mathbb{Z}$-graded structure induced by $\LL^+$.  In addition, we can regard $V$ as a finite graded $\LL_0$-module. Let $V_m$ be the highest non-zero components of $V$.  Since $\LL$ is nilpotent, $V_m\neq {{\LL_0}}_{(\lambda)}V_m$. Hence there exists some $y\in \LL_m$ but not belong to $[{(\LL^++\LL_0)}_{(\lambda)}\LL^+]$.  Let  $\LL_{>n}$ (respectively, $\LL_{<n}$) is the $\C[\partial]$-submodule of $\LL$ with degree greater than $n$(less than n).  
     Set \[S=\{n\in \N| y\in [{\LL_{>n}}_{(\lambda)} \LL_{<m-n}]+[{(\LL^++\LL_0)}_{(\lambda)}\LL^+]\}.\] 
        Since $\LL$ is perfect, $S$ is not empty.
  Let $n_0=min S$. Since $\LL_{n_0}\subset [{\LL^+}_{(\lambda)} \LL^+]$, there exists $0<n_1<n_0$ such that $\LL_{n_0}\subset [{\LL_{n_1}}_{(\lambda)} \LL_{n_0-n_1}]$. By Jocobi-identity,
  we can show that 
  \[  y\in [{\LL_{>n_1}}_{(\lambda)} \LL_{<m-n_1}]+[{(\LL^++\LL_0)}_{(\lambda)}\LL^+],\]
  which contradicts to the minimality of $n_0$.
\end{proof}
\begin{lem}{\label{l3.3}}Suppose that $\LL=\bigoplus_{i\in \mathbb{Z}}\LL_i$ is a rootless finite graded Lie conformal algebra and $M$ is a finite graded weightless $\LL$-module.Then M is nilpotent as a $\LL$-module.\end{lem}
\begin{proof}Since any element of $\LL$ is ad-nilpotent, by \cite[Corollary 8.5]{BDK}, $\LL$ is nilpotent. Hence $\LL$ is solvable. Therefore $M$ has filtration
\[ 0=M_0\subset M_1\subset \cdots \subset M_N=M.\]
such that $M_{i+1}/M_i$ is free rank one for each $i\geq 0$. By the graded structure, $\LL_{>0}$ and $\LL_{<0}$ acts nilpotently on $M$. Since  $M$ is weightless, $\LL_0$ also acts nilpotently on $M$.  Hence  $\LL$ acts nilpotently on $M$. Hence  $M_{i+1}/M_i$ is a trivial $\LL$-module for each $i\geq 0$. Hence $M$ is a nilpotent $\LL$-module. 
\end{proof}
    Let $I(M)$ be  the  $\LL$-submodule of $M$ generated by \[\{v\in M| v\in U(Lie(\LL))Lie(\LL)\cdot v\  \text{and}\  v\  \text{is homogeneuos}. \}\] Clearly, if $M$ is locally nilpotent, then $I(M)=0$.
Define $e(M)=\underline{limits}_{n \mapsto \infty} \frac{1}{n}rank M^n$, where $M^n=\bigoplus_{i\leq |n|}M_i$.
\begin{lem}{\label{l3.4}}Suppose that $\LL=\bigoplus \LL_i$ is a graded Lie conformal algebra of finite type and $M$ is a finitely generated infinite rank  graded $\LL$-module.Then $e(M)\geq 1/d$ for some positive integer $d$.\end{lem}
\begin{proof}Suppose that $\LL$ is generated by $\bigoplus_{i\leq |d|}\LL_i$. Since $M$ is finitely generated, for large enough $p$, $M^p$ generates  $M$. Since $M$ is infinite rank, there exists some $x\in \bigoplus_{i\leq |d|}\LL_i$ and $m\in M_p$ such that $x_{(\M)}m\not\in M^p \cup Tor(M)$. Otherwise, $M^{p+d}$ will be a $\LL$-submodule of $M$, which is impossible by the choice of $M^p$. Hence $rank M^{p+d}\geq rank M^{p}+1$. Therefore, for any $n>p+2d$,
then $n=p+sd+k$ for some integer $s>1$ and $0\leq k<s$. Hence
                                 \[      \frac{1}{n}rank M^n=\frac{1}{p+sd+k}rank M^{p+sd+k}\geq \frac{s}{p+sd+k}\geq \frac{s}{p+(s+1)d}.\]
It implies that $e(M)\geq 1/d$.
\end{proof}

\begin{pro}{\label{p3.5}}Supoose that $\LL=\bigoplus \LL_i$ is a graded Lie conformal algebra of finite type and $M$ is a  graded $\LL$-module. Then there exists some finitely generated $\LL$-submodule $N$ of $M$ such that $M/N$ is locally nilpotent $\LL$-module. \end{pro}
\begin{proof}  Suppose that $\LL$ is generated by $\bigoplus_{i\leq |d|}\LL_i$. Let $E(M)$ be the integer part of $d\cdot e(M)$. We shall prove by induction on $E(M)$. If $E(M)=0$, then any homogenous elements of $M$ must be contained in a finite $\LL$-submodule of $M$. Otherwise, by Lemma \ref{l3.4}, $E(M)\geq 1$. Hence $M$ is locally nilpotent when $E(M)=0$. If $ M$ is not locally nilpotent, then some homogenous elements $v\in V$ exist, such that $N=U(Lie(\LL))v$ is infinite rank. Hence by Lemma {\ref{l3.4}}, $E(N)\geq d\cdot e(N)\geq 1$. Since $rank(M^n)=rank(N^n)+rank({M/N}^n)$, we can obtain that $E({M/N})\leq E(M)-1$.  Thus by induction, we can find a finitely generated $\LL$-submodule $P/N$ such that $\frac{\LL/N}{P/N}\cong \LL/P$ is locally nilpotent. In addition, $P$ is finitely generated since $N$ and $P/N$ is finitely generated.
\end{proof}

\begin{cor}{\label{c3.6}} Let $\LL$,$M$ be as in Proposition {\ref{p3.5}},then $I(M)$ is a finitely generated $\LL$-module.\end{cor}
\begin{proof} Applying $I(M)$ to Proposition \ref{p3.5}, there exists some finitely generated $\LL$-submodule  $N$ of $I(M)$ such that $I(M)/N$ is locally nilpotent.
Notice that $I(I(M))=I(M)$. Hence, we have $0=I(I(M)/N)=I(M)/N$. Thus $I(M)$ is finitely generated as an $\LL$-module.  
\end{proof}

\begin{theorem}{\label{t3.7}}Suppose that  $\LL=\bigoplus_{i\in \mathbb{Z}} \LL_i$ is a uniformly bounded simple Lie conformal algebra. Then $\LL$ is not rootless and of finite type. In particular, $\LL_0\neq \{0\}$.\end{theorem}
\begin{proof}
For each homogeneous element $x\in \LL$, the ideal generated by $[\LL_{(\lambda)}x]$ must contain $x$ by the simplicity of $\LL$. Thus there exists some finite type graded subalgebra $H$ of $\LL$  such that $x\in ad(U(Lie(H))Lie(H))\cdot x$.  Consider $\LL$ as an $H$-module by adjoint action. By Proposition \ref{p3.5}, there exists some finitely generated $H$-submodule $P$ of $\LL$ such that $\LL/P$ is a locally nilpotent $H$-module. Let $\mathcal{I}$ be the $H$-submodule $I(\LL)$ of $\LL$. Next, 
Since $\LL/P$ is locally nilpotent, for any  homogeneous elements $v\in \LL/P$, there exists a finite nilpotent $H$-submodule  $V/P$ of $\LL/P$ contains $v$, \[ P=V_0\subset V_1\subset \cdots \subset V_n=V,\]
where   $V_{i+1}/V_i$ is a trivial $H$-module for each $i\geq 0$.  For any $x \in \mathcal{I}$, we have $u \in U(Lie(H))Lie(H)$ such that $u\cdot x=x$. Hence
for each $j\geq 0$, \[ u\cdot (v_{(j)} x)\equiv v_{(j)} (u\cdot x)\ mod [{V_{n-1}}_{(\lambda)} \mathcal{I}]\equiv v_{(j)} x  mod [{V_{n-1}}_{(\lambda)} \mathcal{I}].\]
Applying $u$ on  $v_{(j)} x$ for $n$ times, we can see that 
\[ [\LL_{(\lambda)}\mathcal{I}]\subset [P_{(\lambda)}\mathcal{I}]+\mathcal{I}.\]
Let $X=\sum_{n\geq 0} ad^n(P_{(\lambda)})\mathcal{I}$. Since $x\in \mathcal{I}$, $X$ is a non-zero graded ideal of $\LL$. Hence $X=\LL$.  From 
Corollary \ref{c3.6}, $P$ and $\mathcal{I}$ are both finitely generated $H$-modules. Hence $\LL$ is a finitely generated $H$-module. Therefore $\LL$ is finite type. Hence by  Proposition \ref{p3.2},  $\LL$ is not rootless.  
\end{proof}

\section{Simple graded Lie conformal algebras with bound one}
\subsection{Preliminaries}

In this section, we aim to classify the following simple graded Lie conformal algebra $\mathcal{L}=\bigoplus_{i\in \Z}\mathcal{L}_i$ such that
$rank{\mathcal{L}_i}\leq 1$ for each $i\in\Z$.  From Theorem {\ref{t3.7}},  we obtain that $\LL_0\neq \{0\}$. By \cite{DK}, $\LL_0$ is either abelian or isomorphic to $Vir$. 
The classification was completed in \cite{X} when $\LL_0\cong Vir$. Hence we will focus on the case that $\LL_0$ is abelian. 
Let  $\mathcal{L}=\bigoplus_{i\in \Z}\mathcal{L}_i$ be a graded Lie conformal algebra.
The following notation will be used in the sequel, 
\begin{enumerate}[fullwidth,itemindent=0em,label=(\arabic*)]
	\item[$\bullet$] $Supp(\mathcal{L}):=\{i\in \mathbb{Z}| rank(\LL_{i})=1\}.$ 
	\item[$\bullet$] For each $i\in Supp(\mathcal{L})$, we use $L_i$ to denote a $\mathbb{C}[\p]$-module basis of $\LL_{i}$.
	\item[$\bullet$] For each $i,j$,$i+j\in Supp(\mathcal{L})$, we use $p_{i,j}(\partial,\lambda)$ as the structure polynomial between $L_i$ and $L_j$ for some $p_{i,j}(\p,\M)\in \mathbb{C}[\p,\M]$, that is $[{L_i}_\lambda {L_j}]=p_{i,j}(\partial,\lambda)L_{i+j}$. 
	\item[$\bullet$] $p_{i,j}(\p,\M)$ will be abbreviated as $p_{i,j}(\p)$(respectively, $p_{i,j}(\M)$) when $p_{i,j}(\p,\M)\in \C[\p](respectively, \C[\lambda])$. 
	\item[$\bullet$] $Supp_1(\mathcal{L}):=\{j\in Supp(\mathcal{L})|p_{-j,j}(\p,\M)\neq 0\  and\  p_{0,j}(\p,\M)\neq 0\}.$
	\item[$\bullet$] $\mathcal{L}[k]$: the subalgebra generated by $\LL_{k}$,$\LL_{-k}$,$\LL_{0}$  for $k\in Supp_1(\mathcal{L})$.
\end{enumerate}
By skew-symmetric, we know that $p_{i,j}(\p,\M)=-p_{j,i}(\p,-\p-\M)$. In particular, $p_{i,i}(\p,\M)=-p_{i,i}(\p,-\p-\M,)$. Hence $(\p+2\M)|p_{i,i}(\p,\M)$.
In addition,  we have $p_{0,i}(\p,\M)\in \C[\M]$.

\begin{lem}{\label{l4.1}}Suppose that $\LL=\bigoplus_{i\in \mathbb{Z}} \LL_i$ is a  $\Z$-graded Lie conformal algebra. Let $S$ be some subset of $\Z$.  Set $\mathcal{S}=\bigoplus_{i\in S}\LL_i$ and $\mathcal{S}^{\perp}=\bigoplus_{i\not\in S}\LL_i$. Suppose that $ [\mathcal{S}_{(\lambda)}\mathcal{S}^{\perp}]\subset \mathcal{S}^{\perp}$. Then $\mathcal{S}^{\perp}+[{\mathcal{S}^{\perp}}_{(\lambda)} \mathcal{S}^{\perp}]$ is an ideal of $\LL$. In particular, if $\LL$ is simple graded, then
	$\LL=\mathcal{S}^{\perp}+[{\mathcal{S}^{\perp}}_{(\lambda)} \mathcal{S}^{\perp}]$ and $\mathcal{S}\subset[{\mathcal{S}^{\perp}}_{(\lambda)}{\mathcal{S}^{\perp}}]$.
\end{lem}
\begin{proof}Only notice that \[ [{\mathcal{S}^{\perp}}_{(\mu)}[{\mathcal{S}^{\perp}}_{(\lambda)} {\mathcal{S}^{\perp}}]]\subset [{\mathcal{S}^{\perp}}_{(\lambda)} ({\mathcal{S}^{\perp}}+\mathcal{S})]\subset \mathcal{S}^{\perp}+[{\mathcal{S}^{\perp}}_{(\lambda)} \mathcal{S}^{\perp}],\] 
	\[ [\mathcal{S}_{(\mu)}[\mathcal{S}^{\perp}_{(\lambda)} \mathcal{S}^{\perp}]]\subset [[\mathcal{S}_{(\mu)} \mathcal{S}^{\perp}]_{(\lambda+\mu)}\mathcal{S}^{\perp}]+[{\mathcal{S}^{\perp}}_{(\lambda)} [\mathcal{S}_{(\mu)}\mathcal{S}^{\perp}]]\subset \mathcal{S}^{\perp}+[\mathcal{S}^{\perp}_{(\lambda)} \mathcal{S}^{\perp}].\]
\end{proof}

\begin{lem}{\label{l4.2}}
	 Suppose that $k\in Supp(\LL)$ such that $p_{0,k}(\M)=0$. If $\LL$ is simple graded, then there exists some $m\in Supp(\LL)$ such that $p_{0,m}(\M)\neq 0$ and  $p_{m,k-m}(\p,\M)\neq 0$.In particular, $Supp_1(\mathcal{L})$ is not empty;
\end{lem}
\begin{proof}
	 Let  $S=\{i\in Supp(\LL)|p_{0,i}(\p,\M)=0\}$. Set $\mathcal{S}=\bigoplus_{i\in S}\LL_i$
		and apply $\mathcal{S}$ to Lemma \ref{l4.1}. We can find that  $\mathcal{S}^{\perp}=\bigoplus_{i\not\in S}\LL_i$.  Since $\LL$ is simple graded, $L_0$ is not in the center of $\LL$. Hence $\mathcal{S}^{\perp}\neq \{0\}$. Then by  Lemma \ref{l4.1},  $\mathcal{S}\subset [{\mathcal{S}^{\perp}}_{(\lambda)} \mathcal{S}^{\perp}]$. Hence for each $k\in \mathbb{C}$, there exists some $L_m\in \mathcal{S}^{\perp}$ such that $p_{0,m}(\M)\neq 0$ and  $p_{m,k-m}(\p,\M)\neq 0$.
		\end{proof}

\begin{lem}\label{l4.3}
	Suppose that $k,s\in Supp(\mathcal{L})$. If $p_{k,s}(\partial,\lambda)\ne 0$,then
	\begin{equation}\label{eq4-1}
		p_{0,k+s}(\lambda)=p_{0,k}(\lambda)+p_{0,s}(\lambda).
	\end{equation}
\begin{enumerate}[fullwidth,itemindent=0em,label=(\arabic*)]
	\item 
	If $p_{0,k}(\lambda)=0$ and $p_{0,s}(\lambda)\neq 0$, then $p_{k,s}(\partial,\lambda)\in \mathbb{C}[\lambda]$.
	\item
	If $p_{0,k+s}(\partial,\lambda)=0$ and $p_{0,s}(\lambda)\neq 0$, then $p_{k,s}(\partial,\lambda)\in \mathbb{C}[\partial]$.
	\item
	If $p_{0,k}(\partial,\lambda),p_{0,s}(\partial,\lambda),p_{0,k+s}(\partial,\lambda)\ne 0$, then $deg p_{k,s}(\partial,\lambda)\leq 1$.
	In addition, $deg p_{k,s}(\partial,\lambda)=1$ if and only if 
	\begin{align*}
		p_{0,s}(\partial,\lambda)&=c p_{0,k}(\partial,\lambda),\\
		p_{0,k+s}(\partial,\lambda)&=(c+1) p_{0,k}(\partial,\lambda),
	\end{align*}
	In this case,    $p^1_{k,s}(\partial,\lambda)=c'(\partial+(c+1)\lambda)$, where $c'\in{\mathbb{C}}^{*}$ and $c\neq 0,-1$.
\end{enumerate}
\end{lem}
\begin{proof}
	Consider the Jacobi-identity of $L_0$,$L_k$,$L_s$ and $L_{k+s}$,we can find that 
	\begin{equation}\label{eq4-2}
		p_{0,k}(\lambda)p_{k,s}(\partial,\lambda+\mu)=p_{0,k+s}(\lambda)p_{k,s}(\partial+\lambda,\mu)-p_{0,s}(\lambda)p_{k,s}(\partial,\mu).
	\end{equation}
	Write $p_{k,s}(\partial,\lambda)=\sum f_{i}(\lambda){\partial}^i$ for some $f_{i}(\M)\in\C[\M]$ and compare the  highest degree of $\partial$, say $m$, of two sides of Equation \eqref{eq4-2}, we can get that  
	\begin{equation}\label{eq4-3}
		p_{0,k}(\lambda)f_{m}(\lambda+\mu)=p_{0,k+s}(\lambda)f_{m}(\mu)-p_{0,s}(\lambda)f_{m}(\mu).
	\end{equation}
   Comparing the degree of coefficients  before $\mu$ in the Equation \eqref{eq4-3}, we can find that
	\begin{equation*}
		p_{0,k+s}(\lambda)=p_{0,k}(\lambda)+p_{0,s}(\lambda).
	\end{equation*}
	
	\begin{enumerate}[fullwidth,itemindent=0em,label=\rm(\arabic*)]
		\item 
		If $p_{0,k}(\partial,\lambda)=0$ and $p_{0,s}(\lambda)\neq 0$, then
		$p_{0,k+s}(\lambda)=p_{0,s}(\lambda)$. \\
		Plugging $p_{0,k}(\partial,\lambda)=0$ into Equation \eqref{eq4-2}, we can find that 
		\[p_{0,k+s}(\lambda)p_{k,s}(\partial+\lambda,\mu)=p_{0,s}(\lambda) p_{k,s}(\partial,\mu).\]
		Hence $p_{k,s}(\partial+\lambda,\mu)=p_{k,s}(\partial,\mu)$.
		Thus
		$p_{k,s}(\partial,\mu)\in \mathbb{C}[\mu]$.
		\item
		If $p_{0,k+s}(\lambda)=0$, then $p_{0,k}(\lambda)=-p_{0,s}(\lambda)$. Plugging $p_{0,k+s}(\partial,\lambda)=0$ into Equation \eqref{eq4-2}, we have 
		\[p_{0,k}(\lambda)p_{k,s}(\partial,\lambda+\mu)=-p_{0,s}(\lambda)p_{k,s}(\partial,\mu).\]
		Thus $p_{k,s}(\partial,\mu)\in\mathbb{C}[\partial]$.
		\item
		Assume that $degp_{k,s}(\partial,\lambda)>0$.
		Plugging Equation \eqref{eq4-1} into Equation \eqref{eq4-2} and let $\mu=0$, we have 
		\begin{equation}\label{eq4-4}
			p_{0,k}(\lambda)(p_{k,s}(\partial,\lambda)-p_{k,s}(\partial+\lambda,0))=p_{0,s}(\lambda)(p_{k,s}(\partial+\lambda,0)-p_{k,s}(\partial,0)).
		\end{equation}
		where $p_{0,k}(\lambda),p_{0,s}(\lambda)\ne 0$. Thus one can find that both sides of Equation \eqref{eq4-4} are non-zero. Otherwise,
		\begin{equation}\label{eq4-5}
			p_{k,s}(\partial,\lambda)-p_{k,s}(\partial+\lambda,0)=p_{k,s}(\partial+\lambda,0)-p_{k,s}(\partial,0)=0.
		\end{equation}
		Since $degp_{k,s}(\partial,\lambda)>0$, we can find that  $p_{k,s}(\partial,0)\in \C^*$ from Equation \eqref{eq4-5}, which is impossible.
		Since $\lambda$ is a factor of both 	$p_{k,s}(\partial,\lambda)-p_{k,s}(\partial+\lambda,0)$ and $p_{k,s}(\partial+\lambda,0)-p_{k,s}(\partial,0)$, we have 
		\begin{equation}{\label{eq4-6}}
			p_{k,s}(\partial,\lambda)-p_{k,s}(\partial+\lambda,0)=\lambda g(\partial,
			\lambda),\\ p_{k,s}(\partial+\lambda,0)-p_{k,s}(\partial,0)=\lambda h(\partial,\lambda)
		\end{equation} 
		for some $g(\partial,\lambda), h(\partial,\lambda)\in \mathbb{C}[\partial,\lambda]$.
		Plugging Equation \eqref{eq4-6} into Equation \eqref{eq4-4}, one can see that  
	\[ p_{0,k}(\lambda) g(\partial,\lambda)=p_{0,s}(\lambda) h(\partial,\lambda) ,\]
		Since $p_{0,k}(\lambda)$and $h(\partial,\lambda)$ is coprime, $p_{0,k}(\lambda)|p_{0,s}(\lambda)$.
		Similarly, $p_{0,s}(\lambda)|p_{0,k}(\lambda)$.
		Thus $p_{0,s}(\lambda)=cp_{0,k}(\lambda)$ and 	$p_{0,k+s}(\lambda)=(c+1)p_{0,k}(\lambda)$ for some $c\neq 0,-1 $.       
	 Now Equation \eqref{eq4-4} and Equation \eqref{eq4-2} will be  translated as  
		\begin{equation}\label{eq4-7}
			p_{k,s}(\partial,\lambda)=(c+1)p_{k,s}(\partial+\lambda,0)-c p_{k,s}(\partial,0),
		\end{equation}	
		and
		\begin{equation}\label{eq4-8}
			p_{k,s}(\partial,\lambda+\mu)=(c+1)p_{k,s}(\partial+\lambda,\mu)-c p_{k,s}(\partial,\mu).
		\end{equation}
		Plugging Equation \eqref{eq4-7} into Equation \eqref{eq4-8}, we will find that
		\begin{equation}\label{eq4-9}
			p_{k,s}(\partial+\mu,0)-p_{k,s}(\partial,0)=p_{k,s}(\partial+\lambda+\mu,0)-p_{k,s}(\partial+\lambda,0).
		\end{equation}
		Let $n=degp_{k,s}(\partial,\lambda)$.  Then from Equation \eqref{eq4-9}, one can obtain that 
		\begin{equation}\label{eq4-10} (\partial+\mu)^n-{\partial}^n=(\partial+\lambda+\mu)^n-(\partial+\lambda)^n. 
		\end{equation}
        Equation \eqref{eq4-10} forces that $n=1$. In addition, Equation \eqref{eq4-7} tell us that 
		\[ p^1_{k,s}(\partial,\lambda)=c'(\partial+(c+1)\lambda) ,\]
		where $c'\in\mathbb{C}^*$ and $c\neq 0,-1$.
	\end{enumerate}
\end{proof}

\begin{rem}When $\LL_0$ is not abelian, the relations between $p_{0,k}(\p,\M)$, $p_{0,s}(\p,\M)$  and  $p_{0,s+k}(\p,\M)$  is more complicated. See \cite{X}.
	\end{rem}   
	\begin{cor}\label{c4.5}
		 Suppose that $\LL$ is simple graded and  $k \in Supp_1(\LL)$. Then  for each $j\in Supp(\LL)$, $ p_{0,j}(\M)=c_{j} p_{0,k}(\M)$ for some $c_j\in \C$. 	
	\end{cor}
\begin{proof}	Suppose that $p_{0,j}(\M)$ is not associated with  $p_{0,k}(\M)$ for some $j\in Supp(\LL)$. We may assume that $p_{0,j}(\M)\neq 0$. Then, according to Equation \eqref{l4.1}, $p_{0,-k+j}(\M)$ is not zero and is  not associated with $p_{0,-k}(\M)$. Hence $p_{-k,j+k}(\partial,\M)$, $p_{j,k}(\partial,\M)\in \C$. Similarly, $p_{k,-k+j}(\partial,\M)$,$p_{j,-k}(\partial,\M)\in \C$.  By Jacobi-identity of $L_{k}$,$L_{-k}$ and $L_{j}$, we have $p_{0,j}(\M)\in \C^*$. 
Let  $S=\{s\in Supp(\LL) |deg p_{0,s}(\M)=deg p_{0,k}(\M)\}$ and  $\mathcal{S}=\bigoplus_{i\in S}\LL_i$. Then by Equation (\ref{eq4-1}),  $\mathcal{S}$ is an ideal of $\LL$. Thus  $\LL=\mathcal{S}$.	
\end{proof}
	
	\begin{pro}{\label{p4.6}}  For $k \in Supp_1(\mathcal{L})$, then
		\begin{enumerate}[fullwidth,itemindent=0em,label=(\arabic*)]
			\item $-k\in Supp_1(\mathcal{L})$.
			\item  $p_{-k,k}(\p,\M)\in \mathbb{C}[\partial]$.
			\item  $deg p_{-k,k}(\p)+deg p_{0,k}(\M)\leq 2.$ 
		\end{enumerate}
	\end{pro}
	\begin{proof}
		(1)(2) can be obtained from Lemma {\ref{l4.3}} immeadietely. As for (3), let us consider the Jacobi-Identity of $L_k, L_k,L_{-k}$. We have
		\begin{equation}{\label{eq4-11}}
			\begin{split}
				p_{-k,k}(-\M-\mu)p_{0,k}(\M+\mu)=&p_{k,k}(\p+\M,\mu)p_{-k,2k}(\p,\M)\\
				&+p_{-k,k}(\p+\mu)p_{0,k}(-\p-\mu).
			\end{split}
		\end{equation}
		Comparing the terms of  highest degree  on both sides of Equation \eqref{eq4-11}, there exists $g(\p,\M)$,$h(\p,\M)\in \C[\p,\M]$ such that
		\begin{equation}{\label{eq4-12}}
			(\lambda+\mu)^n\pm (\p+\mu)^n=(\p+\M+2\mu)g(\p+\M,\mu)h(\p,\M).
		\end{equation} 
		If $n\geq 3$, then the left side of Equation \eqref{eq4-12} has a factor $(\lambda+\mu\pm\upsilon_i(\partial+\mu))$ where $\upsilon_i$ is a primitive $n$-th unit root. However, this factor can not occur in the right-hand sides of Equation \eqref{eq4-12}. Thus 
		\[deg p_{-k,k}(\p)+deg p_{0,k}(\M)=n\leq 2.\]
	\end{proof}
\begin{dfn} We say that $\mathcal{L}$ is  non-integral. If there exists $a\neq 0,b\in \Z$ and $ f(\lambda)\neq 0$, $g(\lambda)\in \C[\M]$ and $N\in \Z^+$ such that for  
	     \[ p_{0,an+b}(\M)=f(\M)n+g(\M),   \text{ for each}\  n\geq N.  \]
	     Otherwise we say  $\mathcal{L}$ is integral.
	     \end{dfn}
	     Suppose that $\LL=Cur\LL^0$. Then $\LL$ is integral if and only if $\LL^0$ is integral. In addition, one can see that \[M(1),\  M(2),\  ECL_2(s),\  SCL_2(0,s),\ CL_2(0,s), \]
	      are non-integral.

\subsection{Lie Conformal Algebras of Class $\mathcal{V}$}

     We say a   $\Z$-graded Lie conformal algebra $\mathcal{L}=\bigoplus_{i\in \Z}\mathcal{L}_i$ is of class $\mathcal{V}$ if
\begin{enumerate}[fullwidth,itemindent=0em,label=(\arabic*)]
	\item[$({\bf C1})$]$rank \LL_i\leq 1$ for each $i \in \Z$. $rank \LL_1=rank \LL_{-1}=rank \LL_0=1$. 
	\item[$({\bf C2})$]$\LL$ is generated by $\LL_1$,$\LL_0$ and $\LL_{-1}$. 
	\item[$({\bf C3})$]  $[{\LL_{-1}}_\lambda \LL_1]\neq 0,\   [{\LL_{0}}_\lambda \LL_1]\neq 0,  [{\LL_0}_\lambda \LL_{0}]=0.$ 
\end{enumerate}
By Proposition {\ref{p4.6}}, we can see that 
$deg p_{-1,1}(\partial,\lambda)+deg p_{0,1}(\partial,\lambda)\leq 2$. Consequently, we are going to classify the isomorphism classes of Lie conformal algebras within the class $\mathcal{V}$ in accordance with the diverse values of  
 $deg p_{-1,1}(\partial,\lambda)$ and  $deg p_{0,1}(\partial,\lambda)$.

\begin{lem}\label{l4.9}
	Suppose that $degp_{-1,1}(\partial,\lambda)=degp_{0,1}(\partial,\lambda)=1$. Then $Supp(\mathcal{L})=\Z$. Moreover, by suitable choice of $\C[\partial]$-basis $\{L_i\}$,  for each $m,n\in \Z$,  we have \begin{equation}{\label{eq4-13}} p^1_{m,n}(\partial,\lambda)=m\partial+(m+n)\lambda, 
	\end{equation}
\end{lem}
\begin{proof}By exchanging the roles of negative gradation and positive gradation, it suffices to consider the case of  $p_{m,n}(\p,\M)$ for $n\geq 0$. 
	According to Proposition \ref{p4.6}, through an appropriate selection of $L_0$, we can assume that $p^1_{0,1}(\partial,\lambda)=\M$. Similarly, by making a proper choice of $L_{-1}$, we may assume that $p^1_{0,-1}(\partial,\lambda)=\M$. Let us take inductions. 
  Assume that $p^1_{1,k-1}(\p,\M)$ and $p^1_{-1,k}(\p,\M)$ has the form  {\eqref{eq4-13}}  when $n=k>0$.
Consider the Jacobi-identity of $L_{-1}$, $L_1$, $L_k$, we have 
  \begin{equation}\label{eq4-15}
  	\begin{split}
  		k(\lambda+\mu)(\lambda+\mu)=&p^1_{1,k}(\partial+\lambda,\mu)p^1_{-1,k+1}(\partial,\lambda)\\
  		&-(-\partial-\mu+(k-1)\lambda)(\partial+k\mu),
  	\end{split}
  \end{equation}
 Hence 
  	\[p^1_{1,k}(\partial+\lambda,\mu)p^1_{-1,k+1}(\partial,\lambda)=(\partial+\lambda+(k+1)\mu)(-\partial+k\lambda).\]
 It implies that 
  \[ p^1_{1,k}(\partial,\lambda)=c_{1,k}(\partial+(k+1)\lambda), \]
  \[ p^1_{-1,k+1}(\partial,\lambda)=c_{-1,k+1}(-\partial+k\lambda), \]
 for some $c_{1,k}$, $c_{-1,k+1}\in \C$ and $c_{1,k}c_{-1,k+1}=1$. By suitable choice of $L_{k+1}$, we may assume that  $c_{1,k}=1$. Hence $c_{-1,k+1}=1$.
 In particular, one can see that $p_{1,n}(\p,\M)\neq 0$ for each $n\geq 0$. Similarly, $p_{-1,-n}(\p,\M)\neq 0$ for each $n\geq 0$. Hence  $Supp(\mathcal{L})=\Z$ and by Lemma \ref{l4.3}, $p_{0,n}(\M)=np_{0,1}(\M)=n\M$ for each $n\geq 0$. 
  
  As for the case $m,n\geq 0$, by using the induction on $m$ for  $p^1_{m,n}(\p,\M)$, we can obtain the form \eqref{eq4-13}  from the Jacobi-Identity of $L_1$,$L_m$ and $L_n$.  
  
In conclusion, for any $ m,n\in \mathbb{Z}$, we have
  \[ p^1_{m,n}(\partial,\lambda)=m\partial+(m+n)\lambda. \]
\end{proof}

\begin{pro}\label{p4.10}
	Suppose that $degp_{-1,1}(\partial,\lambda)=1$, $degp_{0,1}(\partial,\lambda)=1$. Then $\mathcal{L}\cong CL_2(0,s) $ or $CL_3(s)$.
\end{pro}
\begin{proof}
	\rm  By Lemma \ref{l4.9}, we may assume that for any $m$,$n\in \Z$.
	\[p_{m,n}(\partial,\lambda)=m\partial+(m+n)\lambda+d_{m,n},\]
where $d_{m,n}\in \mathbb{C}$.
	\par\rm  Let us determine the value of $d_{m,n}$. Firstly, from Lemma \ref{l4.3}, we have $d_{0,n}=nd_{0,1}$ for each $n\in \mathbb{Z}$. Next, we shall determine the values of  $d_{-1,n}$ and $d_{1,n}$ for each $n\in \Z$.
	For each $n\in \Z$, consider the Jacobi-identity of $L_{-1}$, $L_{1}$ and $L_{n}$, we have  
	\begin{equation}\label{eq4-20}
		\begin{split}
			&(\partial+\lambda+(n+1)\mu+d_{1,n})(-\partial+n\lambda+d_{-1,n+1})\\
			=&(\lambda+\mu+d_{-1,1})(n(\lambda+\mu)+d_{0,n})\\
			&+(-\partial+(n-1)\lambda-\mu+d_{-1,n})(\partial+n\mu+d_{1,n-1}). 	
		\end{split} 
	\end{equation}
	Comparing the coefficients before $\partial$, $\lambda$ and constant terms in Equation \eqref{eq4-20}, we can find that 
	\begin{align}
		&d_{-1,n+1}-d_{-1,n}=d_{1,n}-d_{1,n-1}, \label{eq4-21}\\
		&d_{-1,n+1}+nd_{1,n}=d_{0,n}+nd_{-1,1}+(n-1)d_{1,n-1}, \label{eq4-22}\\
		&d_{1,n}d_{-1,n+1}=d_{-1,1}d_{0,n}+d_{-1,n}d_{1,n-1}. \label{eq4-23}
	\end{align}
	From Equation \eqref{eq4-21} and Equation \eqref{eq4-22}, for each $n\in \Z$, we have
	\begin{equation}\label{eq4-24}
		d_{1,n}=nd_{0,1}+nd_{-1,1}-nd_{-1,n+1}+(n-1)d_{-1,n}.
	\end{equation}

	From Equation \eqref{eq4-24} and Equation \eqref{eq4-21}, we have
	\begin{equation}\label{4-25}
		(2n-1)d_{-1,n}+d_{0,1}+d_{-1,1}=(n-2)d_{-1,n-1}+(n+1)d_{-1,n+1}.
	\end{equation}

	Using inductions, It is not hard to see that for any  $n>1$ or $n\leq -1$,
	\begin{equation}\label{eq4-26}
		d_{-1,n}=\frac{n+1}{3}(d_{0,1}+d_{-1,1}),
	\end{equation}
Plugging Equation \eqref{eq4-26} into Equation \eqref{eq4-24}, for each $n\geq 1$ or $n< -1$, we have 
	\begin{equation}\label{4-27}
		d_{1,n}=\frac{n-1}{3}(d_{0,1}+d_{-1,1}).
	\end{equation}
	\par Let us consider $d_{m,n}$ for general $m,n\in\mathbb{Z}$. Consider the Jacobi-identity of $L_1$, $L_m$, $L_n$
	\begin{equation}\label{eq4-28}
		\begin{split}
			&(m\partial+n\lambda+(m+n)\mu+d_{m,n})(\partial+(m+n+1)\lambda+d_{1,m+n})\\
			=&(m\lambda-\mu+d_{1,m})((m+1)\partial+(m+n+1)(\lambda+\mu)+d_{m+1,n})\\
			&+(\partial+(n+1)\lambda+\mu+d_{1,n})(m\partial+(m+n+1)\mu+d_{m,n+1}).
		\end{split}
	\end{equation}
	The coefficients before Equation \eqref{eq4-28} tell us that 
	\begin{equation}\label{eq4-29}
		d_{m,n+1}-d_{m,n}=md_{1,m+n}-(m+1)d_{1,m}-md_{1,n}.
	\end{equation}
	First, we deal with the case that $m+n\neq 0$. We may assume that $|m|>|n|$. If $n>0(n<0)$, then $m+n-1(m+n+1)$, $m$, $n$ are both either greater than 1 or smaller than $-1$. Similarly, if $n<0$, then $m+n+1$, $m$, $n$ are both either greater than 1 or smaller than $-1$.   By \eqref{eq4-29} and \eqref{4-27},  it is easy to prove  by using induction that  
			\begin{equation}\label{eq4-30}
			d_{m,n}=\frac{n-m}{3}(d_{0,1}+d_{-1,1})
			\end{equation}
			for each $m+n\neq 0$ and $mn\neq 0$.
If $m+n=0$, it is immeadiete from \eqref{eq4-29} and \eqref{eq4-30} that $d_{m,-m}=md_{1,-1}$ for each $m\in \Z$.
		
	Now plugging Equation \eqref{eq4-26} and  Equation \eqref{4-27} into Equation \eqref{eq4-23}, we have
	\[ \frac{n}{9}(d_{0,1}-2d_{-1,1})(2d_{0,1}-d_{-1,1})=0, \]
	Thus $d_{-1,1}=2d_{0,1}$ or $d_{0,1}=2d_{-1,1}$.
	\begin{enumerate}[fullwidth,itemindent=0em,label=(\roman*)]
		\item
		If $d_{-1,1}=2d_{0,1}$, then $d_{m,n}=(n-m)d_{0,1}$ for each $ m,n\in \mathbb{Z}$. In this case,
		\[[{L'_m}_\lambda {L'_n}]=(m\partial+(m+n)\lambda+(n-m)d_{0,1})L'_{m+n},\]
		for each $m,n\in\mathbb{Z}$. Thus $\mathcal{L}\cong CL_2(0,-d_{0,1})$.
		\item
		If $d_{0,1}=2d_{-1,1}$, then we have
		\begin{equation*}
			\begin{split}
				d_{m,n}=\left \{
				\begin{array}{ll}
					-md_{-1,1},&m+n=0,\\
					2nd_{-1,1},&m=0,\\
					-2md_{-1,1},&n=0,\\
					(n-m)d_{-1,1},&\text{else}.
				\end{array}
				\right. 
			\end{split}
		\end{equation*}
		One can check directly that $\mathcal{L}\cong CL_3(-d_{-1,1})$.
	\end{enumerate}
\end{proof}

\begin{cor}\label{c4.11}
	If $degp_{-1,1}(\partial,\lambda)=2$, then $\mathcal{L}\cong ECL(s)$.
\end{cor}

\begin{proof}
	If $degp_{-1,1}(\partial,\lambda)=2$, then $degp_{0,1}(\partial,\lambda)=0$ from Proposition \ref{p4.6}. Assume that
	\[ p_{-1,1}(\partial,\lambda)=a(\partial-x)(\partial-y) ,\]
	where $a,x,y\in \mathbb{C}$.
Now let $\mathcal{L}'$ be a subalgebra of  $\mathcal{L}$ generated by 
	$(\partial-y)L_0$, $L_1$, $L_{-1}$. Let $L'_0:=(\partial-y)L_0$, $L'_1=L_1$, $L'_{-1}=L_{-1}$.
	Then  \[[{L'_{-1}}_\lambda{L'_1}]=[{L_{-1}}_\lambda{L_1}]=p_{-1,1}(\partial,\lambda)L_0=p_{-1,1}(\partial,\lambda)\frac{1}{\partial-y}L'_0, \]
	
	Hence from Proposition \ref{p4.10}, we can find that $\mathcal{L}'\cong CL_2(0,s)$
	or $CL_3(s)$ for some $s\in \mathbb{C} $. Hence $\mathcal{L}\cong ECL(s)$.
\end{proof}

\begin{cor} {\label{c4.12}}Suppose that   $deg p_{0,1}(\p,\M)=2$. Then $\LL\cong SCL_2(0,s)$.
\end{cor}
\begin{proof}By Lemma {\ref{l4.3}}, $p_{0,n}(\p,\M)=np_{0,1}(\p,\M)$ for each $n\in Supp(\LL)$. Suppose that $p_{0,1}(\p,\M)=a(\M+b)(\M+c)$ for some $a,b,c\in \C$.  Let $\mathcal{W}$ be a free $\Z$-graded $\C[\p]$-modules such that $Supp(\mathcal{W})=Supp(\LL)$. We use $W_i$ to denote the $\C[\p]$-basis of $i$-th component of $\mathcal{W}$ for each $i\in Supp(\mathcal{W})$. We define a  $\M$-brackets $[\cdot_\M \cdot]$ on $\mathcal{W}$ from the structure polynomial of $\LL$ as follows: for each  $i, j\in Supp(\mathcal{W})$,
	\begin{align*}
		\begin{split} 
			[{W_{i}}_\lambda W_{j}]=\left\{ 
			\begin{array}{ll}
				(-\p+b)p_{i,j}(\p,\M)W_0,& i+j=0,\\
				\frac{p_{i,j}(\M)}{\M+b}W_{j},&i=0,\\
				-\frac{p_{i,j}(-\M-\p)}{(-\M-\p+b)}W_{i},&j=0,\\
				p_{i,j}(\p,\M)W_{i+j},&else.
			\end{array}
			\right.
		\end{split}
	\end{align*}
	One can check straightly that $\mathcal{W}$ is a Lie conformal algebra. In addition, $\mathcal{W}$ is generated by $W_1,W_0,W_{-1}$ and $deg (-\p+b)p_{-1,1}(\p,\M)=deg \frac{p_{0,1}(\M)}{\M+b}=1$ by the construction of $\mathcal{W}$. Thus $\mathcal{W}\cong CL_2(0,s)\  \text{or}\  CL_3(s)$ for some $s\in \C$. On the other hand, $\LL$  can be embedded into  $\mathcal{W}$  by mapping $L_0$ to $(-\p+b)W_0$ and $L_i$ to $W_i$ for each $i\neq 0$. Hence  $\LL\cong SCL_2(0,s)$.
\end{proof}

	\begin{pro}\label{p4.13}
	Assume that $p_{-1,1}(\partial,\lambda), p_{0,1}(\partial,\lambda)\in \mathbb{C}^*$, then
	\begin{align*}
		\begin{split}
			\mathcal{L}\cong \left \{ 
			\begin{array}{ll}
				Cursl(2,\mathbb{C}),& p_{1,1}(\partial,\lambda)=p_{-1,-1}(\partial,\lambda)=0,\\
				M(2),& p_{1,1}(\partial,\lambda)\neq 0, p_{-1,-1}(\partial,\lambda)\neq 0,\\
				M(1),& \text{else}.
			\end{array}
			\right.
		\end{split}
	\end{align*}
\end{pro}
\begin{proof}
	\begin{enumerate}[fullwidth,itemindent=0em,label=(\arabic*)]
		\item [{\bf Case 1.}]
		Suppose that $p_{1,1}(\partial,\lambda)=p_{-1,-1}(\partial,\lambda)=0$. It is obvious that $\mathcal{L}\cong Cursl(2,\mathbb{C})$.
		\item [{\bf Case 2.}]
		Assume that $p_{1,1}(\partial,\lambda)\ne 0, p_{-1,-1}(\partial,\lambda)=0$. Consider the  Jacobi-identity of  $L_{-1}$, $L_1$ and  $L_1$
		\begin{equation}\label{eq4-32}
			\begin{split}
			p_{-1,1}(-\lambda-\mu)p_{0,1}(\lambda+\mu)=&p_{1,1}(\partial+\lambda,\mu)p_{-1,2}(\partial,\lambda)\\
			&+p_{-1,1}(\partial+\mu)p_{0,1}(-\partial-\mu).
			\end{split}
		\end{equation}
	Hence $p_{1,1}(\partial+\lambda,\mu)p_{-1,2}(\partial,\lambda)=0$. It implies that $p_{-1,2}(\partial,\lambda)=0$. Next consider the  Jacobi-identity of  $L_{-1}$, $L_1$, $L_2$, we have 
		\begin{align*}
			p_{1,2}(\partial+\lambda,\mu )p_{-1,3}(\partial,\lambda)=&p_{-1,1}(-\lambda-\mu,\lambda)p_{0,2}(\partial,\lambda+\mu)\\
			&+p_{-1,2}(\partial+\mu,\lambda)p_{1,1}(\partial,\mu).
		\end{align*}
		Thus
		\[p_{1,2}(\partial+\lambda,\mu )p_{-1,3}(\partial,\lambda)=p_{-1,1}(-\lambda-\mu,\lambda)p_{0,2}(\partial,\lambda+\mu).\]
		Since $p_{0,2}(\lambda)=2p_{0,1}(\lambda)\in \mathbb{C}^*$, we have $p_{1,2}(\partial,\lambda),p_{-1,3}(\partial,\lambda)\in \mathbb{C}^*$. For each $n\geq 0$,  we can obtain by induction that 
		\[\frac{(n-1)(n+2)}{2}p_{-1,1}(-\lambda-\mu)p_{0,1}(\lambda+\mu)=p_{1,n}(\partial+\lambda,\mu)p_{-1,n+1}(\partial,\lambda),\]
	In particular, for each
	$n\geq2$, we have $p_{1,n}(\partial,\lambda)\in \mathbb{C}^*$ while $p_{-1,n+1}(\partial,\lambda)\in \mathbb{C}^*$ for each $n\geq 0$ and $n\neq 1$. 
Then consider the Jacobi-identity of  $L_1$, $L_1$ and $L_n$, we have
		\begin{equation}\label{eq4-33}
		\begin{split}
			p_{1,n}(\partial+\lambda,\mu)p_{1,n+1}(\partial,\lambda)=&p_{1,1}(-\lambda-\mu,\lambda)p_{2,n}(\partial,\lambda+\mu)\\
			&+p_{1,n}(\partial+\mu,\lambda)p_{1,n+1}(\partial,\mu).
		\end{split}
	    \end{equation}
		From Equation \eqref{eq4-33}, we can obtain that $p_{1,1}(-\lambda-\mu,\lambda)p_{2,n}(\partial,\lambda+\mu)=0$. Since $p_{1,1}(\partial,\lambda)\ne 0$, for each $ n\geq2$, $p_{2,n}(\partial,\lambda)=0$.
		Next we use induction on $m$ for $p_{m,n}(\p,\M)$.
		For each $ n\geq2$ assume that $p_{s-1,n}(\partial,\lambda)=0$ for some $s-1\geq 2$. Considering the Jacobi-identity of $L_1$, $L_{s-1}$ and $L_n$, we have
		\begin{align*}
			p_{s-1,n}(\partial+\lambda,\mu)p_{1,s+n-1}(\partial,\lambda)=&p_{1,s-1}(-\lambda-\mu,\lambda)p_{s,n}(\partial,\lambda+\mu)\\
			&+p_{1,n}(\partial+\mu,\lambda)p_{s-1,n+1}(\partial,\mu).
		\end{align*}
Thus
\[p_{1,s-1}(-\lambda-\mu,\lambda)p_{s,n}(\partial,\lambda+\mu)=0.
\]	
		Since  $p_{1,s-1}(\partial,\lambda)\in\mathbb{C}^*$, we obtain that  $p_{s,n}(\partial,\lambda)=0$. We may assume that $p_{1,1}(\partial,\lambda)=c_{1,1}(\partial+2\lambda)$ for some $c_{1,1}\in \mathbb{C}^*$. Now let
		\begin{align*}
			\begin{split}
				L'_n=\left \{
				\begin{array}{ll}
					\frac{1}{p_{-1,1}(\partial)p_{0,1}(\lambda)}L_{-1},&n=-1,\\
					\frac{-1}{p_{0,1}(\lambda)}L_0,&n=0,\\
					L_1,&n=1,\\
					c_{1,1}L_2,&n=2,\\
					p_{1,n-1}(\partial,\lambda)p_{1,n-2}(\partial,\lambda)\cdots p_{1,2}(\partial,\lambda)c_{1,1}L_{n},&n\geq3.
				\end{array}
				\right.
			\end{split}
		\end{align*}
	One can check directly that $\mathcal{L}\cong M(1)$ under the $\C[\p]$-module basis$\{L'_i|\ i\in \Z_{\geq -1}\}$.
		\item [{\bf Case 3.}]
		If $p_{1,1}(\partial,\lambda)=0$, $p_{-1,-1}(\partial,\lambda)\ne 0$, by exchanging the role of ${L_1}$ and $L_{-1}$, we will back to the Case 2.
		\item [{\bf Case4.}]
		Assume that $p_{1,1}(\partial,\lambda)\ne 0$ and  $p_{-1,-1}(\partial,\lambda)\ne 0$.
		From Lemma \ref{l4.3},\\
		$p_{0,-1}(\lambda)=-p_{0,1}(\lambda)\in \mathbb{C}^*$. Using the similar disscuss as in Case 2,  for each $ n\geq 0$, the Jacobi identity of $L_{-1}$, $L_1$ and $L_n$($L_{-n}$) tells us that
		\[\frac{(n-1)(n+2)}{2}p_{-1,1}(-\lambda-\mu)p_{0,1}(\lambda+\mu)=p_{1,n}(\partial+\lambda,\mu)p_{-1,n+1}(\partial,\lambda),\]
		and
		\[\frac{(n-1)(n+2)}{2}p_{1,-1}(-\lambda-\mu)p_{0,-1}(\lambda+\mu)=p_{-1,-n}(\partial+\lambda,\mu)p_{1,-n-1}(\partial,\lambda).\]
		From the above two equations, for each $n\geq2$, we have \[p_{1,n}(\partial,\lambda), p_{-1,n+1}(\partial,\lambda), p_{-1,-n}(\partial,\lambda), p_{1,-n-1}(\partial,\lambda)\in \mathbb{C}^*.\]
		Similarly, for each $ m,n\geq2$, we have $p_{m,n}(\partial,\lambda)=0$ and $p_{-m,-n}(\partial,\lambda)=0$.
		Further, for each $n\geq 1$, consider the Jacobi-identity of $L_{-1}$, $L_{-1}$ and $L_n$, we have
		\begin{equation}\label{eq4-34}
			\begin{split}
				p_{-1,n}(\partial+\lambda,\mu)p_{-1,n-1}(\partial,\lambda)=&p_{-1,-1}(-\lambda-\mu,\lambda)p_{-2,n}(\partial,\lambda+\mu)\\
				&+p_{-1,n}(\partial+\mu,\lambda)p_{-1,n-1}(\partial,\mu).
			\end{split}
		\end{equation}
		\rm Since  $p_{-1,n+1}(\partial,\lambda)\in\mathbb{C}$ for each $n\geq 0$ and $p_{-1,-1}(\partial,\lambda)\ne0$. Hence for each $n\geq 1$,  $p_{-2,n}(\partial,\lambda)=0$ by Equation \eqref{eq4-34}.
		By using induction, it is not hard to prove that $p_{-m,n}(\p,\M)=0$ for each $m,n\geq 2$. \rm By Lemma \ref{l4.3}(3), we may assume that  $p_{1,1}(\partial,\lambda)=c_{1,1}(\partial+2\lambda)$ and $p_{-1,-1}(\partial,\lambda)=c_{-1,-1}(\partial+2\lambda)$ for some $c_{1,1},c_{-1,-1}\in \mathbb{C}^*$.\\
		Let
		\begin{align*}
			\begin{split}
				L'_{n}=\left \{ 
				\begin{array}{ll}
					\frac{-1}{p_{0,1}(\lambda)}L_0,&n=0,\\
					L_1,&n=1 ,\\
					c_{1,1}L_{2}, &n=2,\\
					p_{1,n-1}(\partial,\lambda)p_{1,n-2}(\partial,\lambda)\cdots p_{1,2}(\partial,\lambda)c_{1,1}L_{n},&n\geq 3,\\
					\frac{1}{p_{-1,1}(\partial)p_{0,1}(\lambda)}L_{-1},&n=-1,\\
					\frac{c_{-1,-1}}{(p_{-1,1}(\partial)p_{0,1}(\lambda))^2}L_{-2},&n=-2,\\
				 \frac{p_{-1,n+1}(\partial,\lambda)p_{-1,n+2}(\partial,\lambda)\cdots p_{-1,-2}(\partial,\lambda)c_{-1,-1}}{\left(p_{-1,1}(\partial)p_{0,1}(\lambda)\right)^{-n}}L_{n},&n\leq -3.
				\end{array}
					\right.
				\end{split}
			\end{align*}
			One can check directly that $\mathcal{L}\cong M(2)$ with  the $\C[\p]$-basis $\{L'_n\}$.
			\end{enumerate}
		\end{proof}

\begin{theorem}{\label{t4.14}} $deg p_{-1,1}(\p)+deg p_{0,1}(\M)\neq 1$. Furthermore, 
	\begin{enumerate}[fullwidth,itemindent=0em,label=(\arabic*)]
	\item If $deg p_{-1,1}(\p)+deg p_{0,1}(\M)=0$, then $\LL$ is of type I. Explicitly, $\LL$ is isomorphic to one of the following:
	\[ Cursl(2,\C),\  M(1),\  M(2) \]
	for some $s\in\mathbb{C}$.
		\item If $deg p_{-1,1}(\p)+deg p_{0,1}(\M)=2$, then $\LL$ is of type II. Explicitly, $\LL$ is isomorphic to one of the following:
		 \[ECL(s),\  SCL_2(0,s),\ CL_2(0,s),\ CL_3(0,s)  \]
		for some $s\in\mathbb{C}$.
	\end{enumerate}
\end{theorem}
\begin{proof}By Proposition {\ref{p4.6}},  $deg p_{-1,1}(\p)+deg p_{0,1}(\M)\leq 2$. From Propositions \ref{p4.10}, \ref{p4.13} and Corollaries  \ref{c4.11}, \ref{c4.12}, it is sufficient to consider the  case that   $deg p_{-1,1}(\p)+deg p_{0,1}(\M)=1$.

	  Suppose that $deg p_{-1,1}(\p)=0$ and $deg p_{0,1}(\p,\M)=1$. Then $p_{-1,1}(\partial)=c_{-1,1}$ and $p^1_{0,1}(\lambda)=c_{0,1}\lambda$ and $p_{1,1}(\partial,\lambda)=c_{1,1}(\partial+2\lambda)$ for some $c_{-1,1},c_{0,1},c_{1,1}\in\mathbb{C}^*$. From the Jacobi identity of $L_{-1}$, $L_{1}$ and $L_1$, we obtain that
	\begin{equation}\label{eq4-35}
		p^1_{1,1}(\partial+\lambda,\mu)p^1_{-1,2}(\partial,\lambda)=c_{-1,1}c_{0,1}(\partial+\lambda+2\mu).
	\end{equation}
Hence $p_{-1,2}(\partial,\lambda)=c_{-1,2}$ for some $c_{-1,2}\in\mathbb{C}^*$ and
		 \[c_{-1,2}c_{1,1}=c_{0,1}c_{-1,1}.\]
However,  the Jacobi-identity of $L_{-1},L_{1},L_{2}$ tells that
	\[p^1_{-1,3}(\p,\M)p^1_{1,2}(\p+\M,\mu)=c_{0,1}c_{-1,1}(\p+2\M+4\mu),\]
	which is impossible. As for the case that $deg p_{-1,1}(\p)=1$ and  $deg p_{0,1}(\p,\M)=0$. Consider the $\M$-brackets between $p_{-1,1}(\p)L_0$, $L_{1}$ and $L_{-1}$. Then it can be reduced to the former case. In conclusion,  the sum of  $deg p_{-1,1}(\p)$ and $deg p_{0,1}(\M)$ can not be 1.                                                                                                                                                                                                                                                                                                                                                                                                                                                                                                                                                                                                                                                                                                                                                                                                                                                                                                                                                                                                                                                                                                                                                                                                                                                                                                                                                                                                                                                                                                                                                                                                                                                                                                                                                                                                                                                                                                                                                                                                                                                                                                                                                                                                                                                                                                                                                                                                                                                                                                                                                                                                                                                                                                                                                                                                                                                                                                                                                                                                                                                                                                                                                                                                                                                                                                                                                                                                                                                                                                                                                                                                                                                                                                                                  
\end{proof}

\subsection{$\LL$ is non-integral}


 \begin{pro}{\label{p4.15}}Let $p_{-k,k}(\p,\M)\neq 0$ and  $p_{0,k}(\M)=0$ for some $k,-k\in Supp(\LL)$. Assume that $p_{0,s}(\p,\M)\neq 0$ for some $s\in Supp(\LL)$.  Then one of the following two cases holds:

	\noindent(i) For each $n\geq 0$, $p_{k,nk+s}(\p,\M)\neq 0$.\ \  (ii)for each $n\geq 0$, $p_{-k,-nk+s}(\p,\M)\neq 0$.
\end{pro}
\begin{proof}

	Assume that $(i)$ is not hold. Then there exists some positive integer $n$ such that $p_{k,kn+s}(\p,\M)=0$ and $p_{k,kj+s}(\p,\M)\neq 0$ for each $0\leq j<n$. Then by Lemma \ref{l4.3}, $p_{0,kj+s}(\p,\M)=p_{0,s}(\p,\M)$ for each $0\leq j\leq n$. 
	
	Consider the Jacobi-identity of $L_k$,$L_{-k}$ and $L_{kn+s}$, it is easy to see that    \[ p_{-k,k}(-\M-\mu,\M)p_{0,s}(\lambda+\mu)=-p_{k,(n-1)k+s}(\p,\mu)p_{-k,nk+s}(\p+\mu,\M).\]
	Thus $p_{0,s}(\lambda+\mu)\in \C^*$, $p_{k,(n-1)k+s}(\p,\M)$ and $p_{-k,nk+s}(\p+\mu,\M)\in \C[\M]$ by Lemma 4.3.
	Now one use induction through the Jacobi-identity of $L_k$,$L_{-k}$ and $L_{kj+s}$ for each $j<n$,  	
	it is easy to see that    \[ -(1+n-j)p_{-k,k}(-\M-\mu,\M)p_{0,s}(\lambda+\mu)=p_{k,(j-1)k+s}(\p,\mu)p_{-k,jk+s}(\p+\mu,\M),\] for each $j\leq n$. In particular,
	$p_{-k,jk+s}(\p,\M)\neq 0$  for each $j\leq 0$.
\end{proof}
\begin{lem}\label{l4.16} If $p_{-k,k}(\partial)\neq 0$ for some $k,-k\in Supp(\LL)$, then $k\in Supp_1(\LL)$.\end{lem}
\begin{proof}Since $\LL$ non-integral, there exists $a\neq 0,b\in \Z$ and $0\neq f(\lambda)$, $g(\lambda)\in \C[\M]$ and $N\in Z^+$ such that for  
	\[ p_{0,aq+b}(\M)=f(\M)q+g(\M),   \text{for each}\  q\geq N.  \]
		In particular, we can take $n> N+|k|$ such that $p_{0,an+b}(\M)\neq 0$.   Suppose that $p_{0,k}(\p,\M)=0$. By Proposition {\ref{p4.15}}, one can see that
	    \[	(ad^{|a|}{L_k}_{(\M)})L_{an+b}\neq \{0\}    \ \ \ \          \text{or}    \ \ 	(ad^{|a|}{L_k}_{(\M)})L_{an+b}\neq \{0\}.\]
Thus
       	      \[ f(\M)n+g(\M)=p_{0,an+b}(\M)=p_{0,a(n+k)+b}(\M)=f(\M)(n+k)+g(\M)\]           \text{or}   	\[f(\M)n+g(\M)=p_{0,an+b}(\M)=p_{0,a(n-k)+b}(\M)=f(\M)(n-k)+g(\M).\]   
It is a  contradiction.
 \end{proof}
\begin{pro}For each $k\in Supp_{1}(\LL)$, $\LL[k]$ has the same type. \end{pro}
\begin{proof}Suppose that $\LL[k_1]\cong ECL(s)$ for some $k_1\in Supp_1(\LL)$. Assume that $\LL[k]$ is of type I  for some $k\in Supp_1(\LL)$.  Then from the Jacobi-identity of $L_{k_1}$,$L_{-k_1}$ and $L_{k}$, one can see that $deg p_{k,k_1}(\p,\M)=1$.  Hence
\begin{equation}{\label{4.35}} p^1_{k,k_1}(\p,\M)=c_{k,k_1}(p_{0,k}(\M)\partial+p_{0,k+k_1}(\M)\M).   
\end{equation}
Plugging \eqref{4.35} into the  Jacobi-identity of $L_{k}$,$L_{-k}$ and $L_{k_1}$, we can see that $p_{0,k}(\M)=p_{0,k+k_1}(\M)$. It implies $p_{0,k_1}(\M)=0$ which is contradict to the choice of $k_1$.	
\end{proof}

\begin{pro}{\label{p4.18}}For each $m,n\in Supp_1(\LL)$, we have $np_{0,m}(\M)=mp_{0,n}(\M)$. \end{pro}
\begin{proof}
(i)$L[m]$ and $L[n]$ are of type I. By Corollary \ref{c4.5}, $p_{0,n}(\lambda)\in \C$ for each $n\in Supp(\LL)$.  It is immeadite from the \cite[Lemma 8]{M2} by consider the structre of $\LL^{0}$.

(ii) $L[m]$ and $L[n]$ are of type II. In this case, we have
		\[np_{0,m}(\M)=p_{0,mn}(\M)=mp_{0,n}(\M).\]
\end{proof}
\begin{cor}{\label{c4.19}}
	There exists some $0\neq f(\M)\in \C[\M]$ such that $p_{0,n}(\M)=nf(\M)$ for each $n\in Supp(\LL)$.
\end{cor}
\begin{proof}
Since $Supp_1(\LL)$ is not empty. Take some $k\in Supp_1(\LL)$. Set \[S=\{n\in Supp(\LL)|np_{0,k}(\M)=kp_{0,n}(\lambda)\}\] and let $\mathcal{S}=\bigoplus_{k\in S}C[\partial]L_k$.  Then by Lemma \ref{l4.1}, $[\mathcal{S}_{(\M)}\mathcal{S}^{\perp}]\subset \mathcal{S}^{\perp}$. Hence by Lemma \ref{l4.2}, there exists some $L_{t}\in \mathcal{S}^{\perp}$ such that $p_{t,-t}(\p)\neq 0$. By Lemma {\ref{l4.16}}, $t\in Supp_1(\LL)$. Hence, by  Proposition {\ref{p4.18}}, $t\in S$, which is a contradiction. Thus $\LL=\mathcal{S}$. Finally, let $f(\M)=\frac{1}{k}p_{0,k}(\M)$. One can check directly that  $p_{0,n}(\M)=nf(\M)$ for each $n\in Supp(\LL)$. 
\end{proof}
\begin{lem}\label{l4.20}Suppose that $L[m]$ and $L[n]$ are of type II for some $m,n\in Supp_1(\LL)$.  if $k_1m+k_2n\neq 0$ for any $k_1,k_2\in \Z$, then  $k_1m+k_2n\in Supp_1(\LL)$.
\end{lem}

\begin{proof}
 Let us consider the Jacobi-identity of $L_{k_1m}$,$L_{-k_1m}$ and $L_{k_2n}$, we will find that $p_{k_1m,k_2n}(\p,\M)\neq 0$. Let us consider Jacobi-identity  the $L_{-k_1m}$,$-L_{k_2n}$ and $L_{k_1m+k_2n}$, we can obtain that 
	\begin{equation}{\label{eq4-36}}
		\begin{split}  
		&p_{-k_1m-k_2n,k_1m+k_2n}(\p,\lambda+\mu)p_{-k_1m,-k_2n}(-\M-\mu,\M)\\=&p_{-k_1m,k_1m}(\p)p_{-k_2n,k_1m+k_2n}(\p+\M,\mu)\\
		&-p_{-k_2n,k_2n}(\p)p_{-k_1m, k_1m+k_2n}(\p+\mu,\M)	
	\end{split}
	\end{equation} 
	 One can see that both sides of  Equation \ref{eq4-36} vanish if and only if  $k_1m+k_2n=0$. Hence if $k_1m+k_2n\neq 0$, then
	$p_{-k_1m-k_2n,k_1m+k_2n}(\p)\neq 0$. Thus $k_1m+k_2n\in Supp_1(\LL)$. 

\end{proof}
\begin{pro}{\label{p4.21}}Suppose that $L[k]$ is of type II for some $k\in Supp_1(\LL)$.  Then $\LL \cong CL_3(s)$ for some $s\in \C$.
\end{pro} 
\begin{proof}
  let $\delta=g.c.d(Supp_1(\LL))$, then  $\delta\in Supp_1(\LL)$ from Lemma \ref{l4.20}. If $Supp(\LL)\neq \delta\Z$, then set $\mathcal{M}=\bigoplus_{i\not\in \delta \Z}\C[\p]L_i$.  Then $\LL=L[\delta]\oplus\mathcal{M}$ and  $[L[\delta]_{(\M)}\mathcal{M}]\subset\mathcal{M}$.  Thus $[M_{(\M)}M]$ is an ideal of $\LL$.   Then  there exists some $L_{t}\in \mathcal{M}$ such that $p_{t,-t}(\p)\neq 0$. Thus $t\in Supp_1(\LL)$, which is a contradiction. Hence $\LL=L[\delta]$. Since $\LL$ is simple graded, $\LL \cong CL_3(s)$.
 \end{proof}
\begin{pro}{\label{p4.22}}Suppose that $L[k]$ is of type I for some $k\in Supp_1(\LL)$. Then the basic Lie algebra $\LL^0$ for $\LL$ is simple graded.  \end{pro}
\begin{proof}Let $I$ be a non-zero proper graded ideal of $\LL^0$. By Corollary \ref{c4.19}, we know that $L_0\not\in I$. Set	
	\begin{equation*}
		\begin{aligned}
				&S:=\{n\in \Z|\text{there exists some}\  L_t\in I\  \text{and}\   k_1,k_2,\cdots,k_n\in \Z\  \text{such that}\\  k_1+k_2&+\cdots+k_n=-t \ 
		\text{and}\  p_{k_i,k_{i-1}+k_{i-2}+\cdots+k_1+t}(\partial,\lambda)\neq 0\  \text{for}\  i=1,2,\cdots n.\}
		\end{aligned}
	\end{equation*}
	   Since $\LL$ is simple graded, $S$ is not empty. Let $n_0$ be the minimal of positive integer of $S$, then there exists $k_1,k_2,\cdots,k_{n_0}\in \Z$ and $L_{t_0}\in I$ satisfy the conditions in set $S$.  Clearly $n_0>1$ otherwise $p_{-t_0,t_0}(\p)\in \C^*$. If $p_{k_1,t_0}(0,0)\in \C^*$, then $L_{k_1+t_0}\in I$, which is contradict to the choice of $n_0$. Hence by Lemma {\ref{l4.3}}, we can have $ p_{k_1,t_0}(\partial,\M)=c_{k_1,t_0}(k_1\p+(k_1+t_0)\M)$  for some $c_{k_1,t_0}\in \C^* $. Hence, let us  first consider the Jacobi-identity of  $L_{k_2},L_{k_1}$ and $L_{t_0}$. we can see that  $p_{k_1+k_2,t_0}(\p,\M)\neq 0$, which is contradict to the minimality of $n_0$.  Thus $I=\{0\}$ and $\LL^0$ is simple graded.
\end{proof}
\begin{theorem}{\label{t4.23}}Suppose that $\LL$ is non-integral. Then $\LL$ is isomorphic to $CL_3(s)$ or $Cur\mathcal{W}_1$ or $Cur\mathcal{W}$, where  $\mathcal{W}$ is  Witt algebra and ${W}_1$ is centerless Virasoro Lie algebra. 
\end{theorem}
\begin{proof}
By Proposition \ref{p4.21}, it is sufficient to consider that  $L[k]$ is of type I for some $k\in Supp_1(\LL)$. By Proposition \ref{p4.22}, $\LL^0$ is isomorphic to either Witt algebra $\mathcal{W}$ or centerless Virasoro Lie algebra $\mathcal{W}_1$.  Suppose that $deg p_{m,n}(\p,\M)>0$ for some $m,n\in Supp(\LL)$. If $m\in Supp_1(\LL)$, then we can have a contradiction to  the Jacobi-identity of $L_m$,$L_{-m}$ and $L_{n}$. If $\LL^0\cong \mathcal{W}$,  for each nonzero $m\in Supp(\LL)$, we have $m\in Supp_1(\LL)$.  Thus we only need to consider the case that $\LL^0\cong \mathcal{W}_1$. In this case, we may assume that $Supp(\LL)=\Z_{\geq -1}$. then applying the adjoint action of $L_{-1}$ for $m+n$ times and using the similar proof as in the Proposition {\ref{p4.22}}, we can also obtain a contradiction. Thus $\LL$ is a current Lie conformal algebra with $\LL^0$ either isomorphic to  $\mathcal{W}$ or to $\mathcal{W}_1$. 
\end{proof}
\subsection{$\LL$ is integral}

 In this part, we will deal with the case that $\LL$ is integral. In this case, from the Theorem \ref{t4.14}, we can see that $L[k]\cong Cursl(2,\C)$ for each $k\in Supp_1(\LL)$.
  \begin{pro}{\label{p4.24}}The following three statements are equivalent.
	\begin{enumerate}[fullwidth,itemindent=0em,label=(\arabic*)]
		\item $\LL$ is integral.
		\item  For each $k\in Supp_1(L)$, the adjoint action of each  $L_k$ and $L_{-k}$ is locally nilpotent.
		\item  For each $k\in Supp(\LL)$, regard $\LL$ as an  $L[k]$-module in terms of adjoint action. Then for each  $j\in Supp(\LL)$,  $L_j$ is contained in a finite  $L[k]$-submodule of $\LL$.
	\end{enumerate}
\end{pro}
\begin{lem}{\label{l4.25}} Suppose that $L[k]\cong Cur \mathfrak{sl}(2,\C)$ for some $k \in Supp_1(\LL)$. Regard $\LL$ as a  $L[k]$-module in terms of adjoint action. Then for each $j\in Supp(\LL)$, $L_j$ is contained in a finite irreducible $L[k]$-submodule of $\LL$. In particular, $p_{k,j}(\p,\M)\in \C$ for each $j\in Supp(\LL)$.
	\end{lem}
	\begin{proof}Since the action of $L_k$ on $\LL$ is locally nilpotent, there exists some $n\geq 0$ such that \[ad^n({L_k}_{(\lambda)})L_j\neq \{0\}\  \text{and}\  ad^{n+1}({L_k}_{(\lambda)})L_j\neq \{0\}.\] Similarly,   \[ad^m({L_{-k}}_{(\lambda)})L_j\neq \{0\}\  \text{and}\  ad^{m+1}({L_{-k}}_{(\lambda)})L_j\neq \{0\}.\] for some $m\geq 0$.  Then $p_{0,nk+j}(\p,\M)$ is a positive integer. Then from the Jacobi-identity of $L_k,L_{-k},L_{nk+j}$,  one can see that 
		\[p_{k,(n-1)k+j}(\p,\M)p_{-k,nk+j}(\p,\M)\in \C^*.\]  Let $s$ be the largest integer smaller than $n$ such that $p_{-k,sk+j}(\p,\M)=0$. Then $p_{0,sk+j}(\M)=-p_{0,nk+j}(\M)$ is negative. If $s>0$, then \[ p_{k,(s-1)k+j}(\p,\M)\neq 0\ \text{but}\   p_{k,(s-1)k+j}(0,0)=0. \] Hence  \[p_{0,(s-1)k+j}(\p,\M)=p_{0,sk+j}(\p,\M)-2<0.\] It implies that module generetad by $L_{(s-1)k+j}$ is non-integral, which is a contradiction.  It implies that $s\leq 0$, then $s=-m$. It implies that \[p_{k,tk+j}(\p,\M)p_{-k,(t+1)k+j}(\p,\M)\in \C^*\] for each $-m+1\leq t\leq n-1$.  It means  that $V=M_{V^0}$(see Example 2.3). Hence $V$ is a finite irreducible $L[k]$-module.
		\end{proof}
		Set $\LL[\Delta]$ be the subalgebra of $\mathcal{L}$ generated by $\{L_0,L_k,L_{-k}|k\in Supp_1(\LL)\}$.
	\begin{lem}{\label{l4.26}}  $\LL[\Delta]=\bigoplus_{i\in S}\LL_i$ for some $S\subset Supp(\LL)$. In addition,  \[[\LL[\Delta]_{(\M)} {\LL[\Delta]}^{\perp}]\subset {\LL[\Delta]}^{\perp}.\] \end{lem}
	\begin{proof}Let $f(\p)L_s\in \LL[\Delta]$ for some non-zero $f(\p)\in \C[\p]$ and $s\in S$. Since $ rank(\LL_i)=1$ for each $i\in S$, $L_s$ must be contained in a non-trivial finite $\LL[k]$-module  for some $k\in Supp_1(\LL)$. Thus the Jacobi-identity of $L_k$, $L_{-k}$ and $L_s$ implies that either
		\[p_{k,s}(\p,\M)p_{-k,k+s}(\p,\M)\in \C^*\  \text{or}\  p_{-k,s}(\p,\M)p_{k,-k+s}(\p,\M)\in \C^*.   \] 
	In both cases, we always have $L_s\in \LL[\Delta]$. 
	Hence  $\LL[\Delta]=\bigoplus_{i\in S}\LL_i$ for some $S\subset Supp(\LL)$.
	
		 Let $k\in Supp_1(\LL)$. Regard $\LL$ as an $\LL[k]$-module by the adjoint action.  By  Lemma \ref{l4.25},  $\LL$ is isomorphic to  $M_{\LL^0}$ as an $\LL[k]$-module.   Since  $L[k]\cong Cursl(2,\C)$ and $\LL^0$ is integrable as an $\mathfrak{sl}(2,\C)$-module,  $\LL^0$ is completely irreducible.  Hence $\LL$ is completely irreducible as an  $L[k]$-module of $\LL$. Since  $\LL[\Delta]$ is   an  $\LL[k]$-submodule,  ${\LL[\Delta]}^{\perp}$ is  an  $\LL[k]$-submodule of $\LL$. Thus  $[\LL[\Delta]_{(\M)} {\LL[\Delta]}^{\perp}]\subset {\LL[\Delta]}^{\perp}$.
\end{proof}
	\begin{pro}{\label{2.2}}$\LL=\LL[\Delta]$.\end{pro}
\begin{proof} $I$ be a graded maximal ideal of $\LL^0$.  Then $\LL^0/I$ is either simple graded or is one dimensional. 
	
	Suppose that $dim(\LL^0/I)=1$ with basis $L_k$ for some $k\in Supp(\LL)$. Clearly, $k\neq 0$ and $L_0\in I$. Hence $\LL[\Delta]\subset I$. Then $[\LL[\Delta]_{(\M)} L_k]=\{0\}$. 
	It impiles that the set \[\mathcal{S}:=\{u\in\LL|[\LL[\Delta]_{(\M)}u ]=\{0\}\}\] is not empty. On the other hand, we have
	   $[\mathcal{S}_{(\lambda)}\mathcal{S}^{\perp}]\subset \mathcal{S}^{\perp}$. Thus by Lemma \ref{l4.2}, there exists some $s\in Supp_1(\LL)$ such that $p_{0,s}(\M)\neq 0$ and  $p_{s,k-s}(0,\M)\neq 0$.  However, by Lemma \ref{l4.3}, we can see that $p_{s,k-s}(\p,\M) \in \C[\p]$ and $p_{s,k-s}(0,0)=0$ for $L_s,L_{k-s}\in I$, which is a contradiction. 
	   
	    Now let us suppose that $\LL^0/I$ is simple graded. In this case $L_0\not\in I$. Hence $C[\p]I \cap  \LL[\Delta]=\{0\}$.	Thus $I\subset  \LL[\Delta]^{\perp}$. By \cite[Lemma 15]{M2}, we know that $\LL^0/I$ is simple graded if and only if $C[\p]I=\LL[\Delta]^{\perp}$.  If  $\LL[\Delta]^{\perp}\neq \{0\}$, By Lemma \ref{l4.1} and Lemma \ref{l4.26}, there exists some $L_t\in I$ such that $ p_{t,-t}(0,\lambda)\neq 0$ and $p_{t,-t}(0,0)=0$ because $L_0\not\in I$.  By Lemma \ref{l4.2}, there exists some $k\in Supp(\LL)$ such that $p_{0,k}(\M)\neq 0$ and $p_{k,t-k}(\p,\M)\neq 0$. Now the Jacobi-identity of  $L_k$,$L_{t-k}$ and $L_{-t}$, we will find that $p_{k,-k}(\p)\neq 0$ and $p_{t-k,k-t}(\p)\neq 0$. Thus $k,t-k\in Supp_1(\LL)$. It implies that $L_t\in \LL[\Delta]$, which is impossible. Hence  $I=\{0\}$ and $\LL=\LL[\Delta]$. 
\end{proof}

\begin{theorem}{\label{t4.28}}If $\LL$ is integral, then $\LL\cong Cur\mathfrak{g}$ for some simple graded Lie algebra $\mathfrak{g}$. 
\end{theorem}
\begin{proof}By Proposition \ref{2.2}, it suffice to prove that  $\LL[\Delta]$ is a current Lie conformal algebra. Let \[S=\{i\in Supp(\LL)|deg p_{j,i-j}(\partial,\lambda)>1,\  \text{for some}\  j\in Supp(\LL)\}.\]Suppose that $S$ is not empty. Suppose that $p_{k,s}(\p,\M)\neq 0$ for some $s\in S$ and $k\in Supp_1(\LL)$. Let $j\in Supp(\LL)$ such that $deg p_{j,s-j}(\partial,\lambda)>1$. Considering the Jacobi-identity of $L_k$, $L_j$ and $L_{s-j}$,  one can see that either $deg p_{k+j,s-j}>0$ or  $deg p_{j,k+s-j}>0$. It implies that $k+s\in S$. Thus  $\mathcal{S}:=\bigoplus_{i\in S}\C[\partial]L_i$ is a nonzero graded ideal of $\LL$, which is a contradiction. Hence  $\LL$ is a current Lie conformal algebra.
	\end{proof}
Now the Theorem {\ref{t01}} can be obtained from Theorem {\ref{t4.23}}, Theorem {\ref{t4.28}} and Proposition {\ref{Xu}}.

\end{document}